\documentclass[10pt, a4paper, reqno]{amsart}
\usepackage[english]{babel}
\usepackage{amsmath,amsfonts,amssymb,amsthm}
\usepackage{amssymb,latexsym, xspace, enumerate}
\usepackage{microtype}

\usepackage{vmargin}
\setmarginsrb{28mm}{20mm}{28mm}{25mm}%
          {10mm}{7mm}{10mm}{10mm}

\numberwithin{equation}{section}

\numberwithin{table}{section}

\usepackage{hyperref}

\usepackage{caption}

\usepackage[all]{xy}
\newtheorem{em-deff}{Definition}[section]
\newtheorem{lemma}[em-deff]{Lemma}
\newtheorem{theorem}[em-deff]{Theorem}
\newtheorem{corollary}[em-deff]{Corollary}

\newtheorem{proposition}[em-deff]{Proposition}
\newtheorem{em-fact}[em-deff]{Fact}
\newtheorem{em-example}[em-deff]{Example}
\newtheorem{claim}[em-deff]{Claim}
\newtheorem{problem}[em-deff]{Problem}

\newtheorem{em-remark}[em-deff]{Remark}

\newenvironment{example}{\begin{em-example} \em }{ \end{em-example}}
\newenvironment{remark}{\begin{em-remark} \em }{ \end{em-remark}}
\newenvironment{definition}{\begin{em-deff} \em }{ \end{em-deff}}
\newenvironment{fact}{\begin{em-fact} }{ \end{em-fact}}

\def\p{\mathbb P}
\def\P{\mathcal P}
\newcommand{\F}{\mathcal F}
\newcommand{\N}{\mathbb N}
\newcommand{\J}{\mathbb J}

\renewcommand{\mod}{\matrm{mod}}
\newcommand{\D}{\mathcal D}

\newcommand{\Z}{\mathbb Z}
\newcommand{\Q}{\mathbb Q}
\newcommand{\R}{\mathbb R}
\newcommand{\C}{\mathbb C}
\newcommand{\T}{\mathbb T}

\DeclareMathOperator{\ent}{ent}

\DeclareMathOperator{\End}{End}

\DeclareMathOperator{\Aut}{Aut}

\def\mod{\mathrm{mod}}
\def\1{(1)}
\def\2{(2)}
\def\3{(3)}
\def\4{(4)}
\def\Pi{\P^{_{<\infty}}}
\def\c{\mathcal C}
\DeclareMathOperator{\flow}{Flow}

\title{Algebraic Yuzvinski Formula}

\author{Anna Giordano Bruno}
\address{Dipartimento di Matematica e Informatica, Universit\`a di Udine, Via delle Scienze, 206 - 33100 Udine}
\email{\tt anna.giordanobruno@uniud.it}
\author{Simone Virili}
\address{Departament de Matematiques, Universitat Autonoma de Barcelona, Edifici C - 08193 Bellaterra (Barcelona)}
\email{\tt simone@mat.uab.cat}

\begin{document}

\maketitle

\begin{abstract}
In 1965 Adler, Konheim and McAndrew defined the topological entropy for continuous self-maps of compact spaces. Topological entropy is very well-understood for endomorphisms of compact Abelian groups. A fundamental result in this context is the so-called Yuzvinski Formula, showing that the value of the topological entropy of a full solenoidal automorphism coincides with the Mahler measure of its characteristic polynomial.

\smallskip
In two papers of 1979 and 1981 Peters gave a different definition of entropy for automorphisms of locally compact Abelian groups. This notion has been appropriately modified for endomorphisms in two recent papers, where it is called algebraic entropy.

\smallskip
The goal of this paper is to prove a perfect analog of the Yuzvinski Formula for the algebraic entropy, namely, the Algebraic Yuzvinski Formula, giving the value of the algebraic entropy of an endomorphism of a finite-dimensional rational vector space as the Mahler measure of its characteristic polynomial.

\smallskip
Finally, several applications of the Algebraic Yuzvinski Formula and related open problems are discussed.
\end{abstract}

\section{Introduction}

Consider a set $X$ and a self-map $T:X\to X$. Denote by $(X,T)$ the discrete-time dynamical system whose evolution law $\N\times X\to X$ is given by $(n,x)\mapsto T^n(x)$.
Depending on the possible structures on $(X,T)$ -- for example when $T$ is a continuous self-map of a topological space $X$, or $T$ is an endomorphism of an Abelian group $X$ -- there exist various notions of entropy, which, roughly speaking, provide a tool to measure the ``disorder" or ``mixing" produced by the action of $T$ on $X$.

In this paper we are mainly concerned with the case when $X$ is a locally compact Abelian (briefly, LCA) group and $T$ is an endomorphism of $X$. By endomorphism of a topological group we always mean continuous endomorphism and with automorphism we intend a group automorphism which is also a homeomorphism. We denote by $\End(G)$  and $\Aut(G)$ respectively  the endomorphisms and the automorphisms  of a given LCA group $G$.  Moreover, we use the standard notations for the reals $\R$, the rationals $\Q$, the integers $\Z$, the natural numbers $\N$ and the positive integers $\N_+$.

\bigskip
In 1965 Adler, Konheim and McAndrew \cite{AKM} introduced the topological entropy for continuous self-maps of compact spaces. In 1971 Bowen \cite{B} gave a different definition of topological entropy for a uniformly continuous self-map $T$ of a metric space $X$, and an alternative description of this topological entropy when the space $X$ is endowed with a $T$-homogeneous measure. In 1974 Hood \cite{hood} noticed that Bowen's definition of topological entropy, as well as its equivalent description, could be extended respectively to uniformly continuous self-maps of uniform spaces and to uniformly continuous self-maps $T$ of locally compact uniform spaces $X$ with a $T$-homogeneous measure. 

Consider a compact space $X$ endowed with its unique admissible uniform structure $\mathcal U$; any continuous self-map $T: X \to X$ is uniformly continuous with respect to $\mathcal U$. By \cite[Corollary 2.14]{DSV}, for such $T$ the topological entropy defined by Hood coincides with the topological entropy defined by Adler, Konheim and McAndrew.

Consider now an LCA group $G$ with a Haar measure $\mu$, and an endomorphism $\phi:G\to G$. In particular, $G$ is a locally compact uniform space when endowed with its canonical left uniformity $\mathcal U$; furthermore, $\phi:(G,\mathcal U)\to (G,\mathcal U)$ is uniformly continuous, and $\mu$ is $\phi$-homogeneous. 
Hence, Hood's extension of Bowen's definition of topological entropy applies to such $G$ and $\phi$, and can be given in the following way (see \cite{DSV}). Denote by $\c(G)$ the family of compact neighborhoods of $0$ in $G$ ordered by inclusion. For every $K\in\c(G)$ and every positive integer $n$,
$$C_n(\phi,K)=K\cap\phi^{-1}K\cap\ldots\cap \phi^{-n+1}K$$
is the {\em $n$-th $\phi$-cotrajectory of $K$}. The \emph{topological entropy} of $\phi$ is
$$h_T(\phi)=\sup\left\{\limsup_{n\to\infty}\frac{-\log\mu(C_n(\phi,K))}{n}:K\in\c(G)\right\};$$
this definition is correct, as Claim \ref{haar} shows. 

The topological entropy is very well-understood on compact groups but only few results are known in the setting of LCA groups. For a comprehensive treatment of these aspects we refer to \cite{DSV}, \cite{WardLN1}, \cite{Walters} and \cite{WardLN}.

\medskip
In the final part of the paper \cite{AKM}, where the topological entropy was defined, also a notion of entropy for endomorphisms of discrete Abelian groups appears. It is based on the following concept of trajectory. Consider a discrete Abelian group $G$, an endomorphism $\phi:G\to G$, a non-empty subset $C$ of $G$, and a positive integer $n$; then
$$T_n(\phi,C)=C+\phi C +\ldots+\phi^{n-1}C$$ 
is the {\em $n$-th $\phi$-trajectory of $C$}. 

Cotrajectories make sense in arbitrary spaces while the concept of trajectory strongly depends on the algebraic operation of the group. This is the reason why we refer to the notions of entropy based on trajectories as {\em algebraic entropies}.

\smallskip
The definition of algebraic entropy of $\phi$ given in \cite{AKM} is
\begin{equation}\label{ent}
\ent(\phi)=\sup\left\{\lim_{n\to \infty}\frac{\log|T_n(\phi,C)|}{n}: \text{ $C$ is a finite subgroup of $G$}\right\} .
\end{equation}
Since a torsion-free Abelian group has no finite subgroups but the trivial one, $\ent(-)$ is always zero on endomorphisms of torsion-free discrete Abelian groups. 
So it is natural to consider
 $\ent(-)$ for endomorphisms of torsion discrete Abelian groups.

\smallskip
In 1974 Weiss \cite{W} studied the basic properties of $\ent(-)$ and connected it with the topological entropy of endomorphisms of profinite Abelian groups via the Pontryagin-Van Kampen duality in the following ``Bridge Theorem''.
For an LCA group $G$, let $\widehat G$ denote the dual group of $G$, endowed with its compact-open topology; moreover, for an endomorphism $\phi:G\to G$, let $\widehat\phi:\widehat G\to \widehat G$ be its dual endomorphism.

\medskip
\noindent {\bf Weiss Bridge Theorem.} {\em Let $G$ be a torsion discrete Abelian group and $\phi:G\to G$ an endomorphism.
Then $\ent(\phi)=h_T(\widehat\phi)$.}

\medskip
In 2009 Dikranjan, Goldsmith, Salce and Zanardo \cite{DGSZ} rediscovered this notion of algebraic entropy and deeply investigated it.
In particular, they characterized $\ent(-)$ as the unique function from the class of the endomorphisms of torsion discrete Abelian groups with target the non-negative reals $\R_{\geq0}$ plus $\infty$, satisfying five very natural axioms (see Section \ref{Appl}).

\medskip
As we already mentioned, $\ent(-)$ has the disadvantage of being trivial on endomorphisms of torsion-free discrete Abelian groups. In 1979 Peters \cite{Pet} proposed an alternative notion of algebraic entropy for automorphisms $\phi$ of discrete Abelian groups $G$, defining 
\begin{equation}\label{hp-def}
h_\infty(\phi)=\sup\left\{\lim_{n\to\infty}\frac{\log|T_n(\phi^{-1},C)|}{n}:C\text{ a finite subset of $G$}\right\}.
\end{equation}
This entropy takes the same values as $\ent(-)$ on automorphisms of torsion discrete Abelian groups but it may be non-zero also in the torsion-free case. 
In the same paper \cite{Pet}, Peters stated some general properties of $h_\infty(-)$ and proved the following Bridge Theorem for automorphisms of countable discrete Abelian groups.

\medskip
\noindent {\bf Peters Bridge Theorem.} {\em Let $G$ be a countable discrete Abelian group and $\phi:G\to G$ an automorphism.
Then $h_\infty(\phi)=h_T(\widehat\phi)$.}

\medskip
Recently Peters' definition of algebraic entropy was appropriately modified in \cite{DG}, replacing in \eqref{hp-def} the trajectories of $\phi^{-1}$ with the trajectories of $\phi$. In this way it was obtained a notion of algebraic entropy for all endomorphisms of discrete Abelian groups; we denote it by $h_A(-)$. Note that $h_A(-)$ has the same values of $h_\infty(-)$ when they both make sense, namely on automorphisms of discrete Abelian groups; moreover, $h_A(-)$ coincides with $\ent(-)$ on endomorphisms of torsion discrete Abelian groups.

\smallskip
In 1981 Peters \cite{Pet1} gave a further generalization of the entropy $h_\infty(-)$ he defined in \cite{Pet}. In fact, using the Haar measure, he introduced a notion of entropy for automorphisms of LCA groups as follows. Let $G$ be an LCA group with a Haar measure $\mu$, and let $\phi:G\to G$ be an automorphism. Then
$$h_\infty(\phi)=\sup\left\{\limsup_{n\to\infty}\frac{\log\mu(T_n(\phi^{-1},C))}{n}:C\in\c(G)\right\} .$$

In \cite{V} the second named author modified the definition of $h_\infty(-)$, in the same way as done in \cite{DG} for the discrete case, obtaining a new notion of algebraic entropy for endomorphisms of LCA groups.
We indicate it  by $h_A(-)$ as it coincides on endomorphisms of discrete Abelian groups with the already defined algebraic entropy from \cite{DG} (see Example \ref{ent=h_A}(b)). For the precise relation between $h_\infty(-)$ and $h_A(-)$ see Remark \ref{h_infty=h_A}.
With the same notations as above, the \emph{algebraic entropy of $\phi$ with respect to $C\in\c(G)$} is
$$H_A(\phi,C)=\limsup_{n\to\infty}\frac{\log\mu(T_n(\phi,C))}{n} ;$$
this definition is correct, as Claim \ref{haar} shows. The \emph{algebraic entropy of} $\phi$ is
$$h_A(\phi)=\sup\{H_A(\phi,C):C\in\c(G)\} .$$

\bigskip
We recall now another concept which plays a fundamental role in the present paper.
Let $N$ be a positive integer, let $f(X)=sX^N+a_1X^{N-1}+\ldots+a_N\in\C[X]$ be a non-constant polynomial with complex coefficients and let $\{\lambda_i:i=1,\ldots,N\}\subseteq\mathbb C$ be of all roots of $f(X)$ (we always assume the roots of a polynomial to be counted with their multiplicity); in particular, $f(X)=s\cdot\prod_{i=1}^N(X-\lambda_i)$. The Mahler measure of $f(X)$ was defined independently by Lehmer \cite{Lehmer} and Mahler \cite{Mahler} in two different equivalent forms.
Following Lehmer \cite{Lehmer} (see also \cite{Ward0}), the \emph{Mahler measure} of $f(X)$ is $M(f(X))=|s|\cdot \prod_{|\lambda_i|>1} |\lambda_i|.$ The \emph{(logarithmic) Mahler measure} of $f(X)$, that is the form that we use in this paper, is $$m(f(X))=\log M(f(X))=\log|s|+\sum_{|\lambda_i|>1}\log|\lambda_i|.$$
The Mahler measure plays an important role in number theory and arithmetic geometry; in particular, it is involved in the famous Lehmer Problem asking whether $\inf\{m(f(X)):f(X)\in\Z[X]\ \text{primitive}, m(f(X))>0\}$ is strictly positive (for example see  \cite{Ward0}, \cite{Hi} and \cite{M}, and for a survey on the Mahler measure of algebraic numbers see \cite{Smyth}).

\medskip
A rational $N\times N$ matrix $M$ has its monic characteristic polynomial $f(X)\in\Q[X]$; we say that $f(X)$ is the characteristic polynomial of $M$ over $\Q$. Let $s$ be the minimum positive integer such that $sf(X)\in\Z[X]$ (i.e., $s$ is the minimum positive common multiple of the denominators of the coefficients of $f(X)$); then we say that $p(X)=sf(X)$ is the characteristic polynomial of $M$ over $\Z$. 

Let $\phi:\Q^N\to \Q^N$ be an endomorphism and consider the $N\times N$ rational matrix $M_\phi$ representing the action of $\phi$ on $\Q^N$ with respect to the canonical base of $\Q^N$ over $\Q$. We call characteristic polynomial $p_\phi(X)$ of $\phi$ over $\Q$ (respectively, over $\Z$) the characteristic polynomial of $M_\phi$ over $\Q$ (respectively, over $\Z$); moreover, by eigenvalues of $\phi$ we mean the eigenvalues  of $M_\phi$. 

\medskip
We recall that a solenoid is a finite-dimensional connected compact Abelian group; so its dual group is a finite rank torsion-free discrete Abelian group, i.e., a subgroup of $\Q^N$ for some positive integer $N$. Moreover, $\widehat \Q^N$ is called full solenoid. With (full) solenoidal endomorphism we mean an endomorphism of a (full) solenoid.  

The action of a continuous endomorphism $\psi:\widehat \Q^N\to \widehat \Q^N$ is represented by an $N\times N$ rational matrix $M_\psi$, which is the transposed of the $N\times N$ rational matrix $M_{\phi}$ representing the action of the dual $\phi=\widehat \psi$ of $\psi$ on $\Q^N$.
The characteristic polynomial $p_\psi(X)$ of $\psi$ over $\Q$ (respectively, over $\Z$) is the the characteristic polynomial of $M_\psi$ over $\Q$ (respectively, over $\Z$); clearly $p_\psi(X)=p_\phi(X)$. Moreover, by eigenvalues of $\psi$ we mean the eigenvalues  of $M_\psi$.

\smallskip
One of the main results about the topological entropy for endomorphisms of compact Abelian groups is the following formula computing the topological entropy of endomorphisms of full solenoids in terms of the Mahler measure of their characteristic polynomial. 

\medskip
\noindent {\bf Yuzvinski Formula. }{\em
Let $N$ be a positive integer and $\phi:\widehat\Q^N\to \widehat\Q^N$ an endomorphism. Then
 $$h_T(\phi)= m(p_\phi(X)),$$
where $p_\phi(X)$ is the characteristic polynomial of $\phi$ over $\Z$.}

\medskip
This nice formula was obtained by Yuzvinski in \cite{Y}.  A different and more conceptual approach to the same result was given by Lind and Ward in \cite{LW}, where they also described in detail the history of the Yuzvinski Formula and related results.

\smallskip
The Yuzvinski Formula has a wide range of applications, as it allows the computation of the topological entropy of solenoidal automorphisms. As a consequence of the results of \cite{LW}, one can also obtain the already known Kolmogorov-Sinai Formula, stating that the topological entropy of a toral automorphism $\phi:\T^N\to \T^N$, which is described by an $N\times N$ matrix with integer coefficients, is $h_T(\phi)=\sum_{|\lambda_i|>1}\log|\lambda_i|$, where $\{\lambda_i:i=1,\ldots,N\}$ are the eigenvalues of $\phi$.

Moreover, the Yuzvinski Formula is one of the eight axioms in Stojanov's characterization of the topological entropy for endomorphisms of compact (not necessarily Abelian) groups obtained in \cite{S}.

\bigskip
The goal of this paper is to provide a completely self-contained proof of the following algebraic counterpart of the Yuzvinski Formula, computing the algebraic entropy of endomorphisms of finite dimensional rational vector spaces in terms of the Mahler measure of their characteristic polynomials.

\medskip
\noindent {\bf Algebraic Yuzvinski Formula. }{\em
Let $N$ be a positive integer and $\phi:\Q^N\to \Q^N$ an endomorphism. Then
 $$h_A(\phi)= m(p_\phi(X)),$$
where $p_\phi(X)$ is the characteristic polynomial of $\phi$ over $\Z$.}

\medskip
At this point the careful reader should have noticed that a proof of the Algebraic Yuzvinski Formula (at least for automorphisms) could be obtained from the classical Yuzvinski Formula applying Peters Bridge Theorem, since 
$h_A(-)$ and $h_\infty(-)$ coincide on automorphisms of discrete Abelian groups. Such approach is not justified, mainly in view of the highly sophisticated proof of Peters Bridge Theorem, heavily using convolutions.
Furthermore, as discussed in detail in \cite{DSV}, some of the proofs in \cite{Pet} and \cite{Pet1} contain some inaccuracy.

On the other hand, our direct proof of the Algebraic Yuzvinski Formula is motivated also by its application for the proof of the results in \cite{DG} about the algebraic entropy of endomorphisms of discrete Abelian groups, not last a generalized version of the Bridge Theorem given in \cite{DG-BT}, deduced from the Algebraic Yuzvinski Formula and Weiss Bridge Theorem, making no recourse to Peters Bridge Theorem. We refer to Section \ref{Appl} for a more exhaustive discussion on the applications of the Algebraic Yuzvinski Formula.

\smallskip
 It is worth mentioning that a first attempt to prove the Algebraic Yuzvinski Formula was done in \cite{Z-yuz}, where several partial results were obtained; some of them were recently used to prove the ``case zero'' of the Algebraic Yuzvinski Formula in \cite{DGZ} with arguments exclusively of linear algebra. Indeed, under the same notations and hypotheses as above, \cite[Corollary 1.4]{DGZ} shows that $h_A(\phi)=0$ if and only if $m(p_\phi(X))=0$.
Moreover, the algebraic counterpart of the Kolmogorov-Sinai Formula was established in the paper \cite{V}, where the algebraic entropy for endomorphisms of LCA groups was introduced.
Indeed, in \cite{V} the algebraic entropy of an endomorphism $\phi:\Z^N\to\Z^N$ was computed proving the Algebraic Kolmogorov-Sinai Formula, that is, $h_A(\phi)=\sum_{|\lambda_i|>1}\log|\lambda_i|$, where $\{\lambda_i:i=1,\ldots,N\}$ are the eigenvalues of $\phi$. The main idea for the proof of this result was to extend $\phi$ to an endomorphism $\Phi$ of $\R^N$ with the same algebraic entropy, and to use there the Haar measure to estimate the growth of the trajectories of $\Phi$ in order to compute its algebraic entropy.
These methods inspired the ones used in this paper for the proof of the Algebraic Yuzvinski Formula. We see in Section \ref{Appl} that it is possible to deduce the Algebraic Kolmogorov-Sinai Formula from the Algebraic Yuzvinski Formula. Finally, note that our proof is independent from the results of  \cite{DGZ}, \cite{V} and \cite{Z-yuz}.

\bigskip
Now we pass to describe the structure and the results of the present paper. Our proof of the Algebraic Yuzvinski Formula relies on many steps and takes most of the paper. On the other hand, several partial results are of their own interest. For example the forthcoming Facts A and C are used in \cite{DGSV}.

\smallskip
Section \ref{Background} is devoted to provide some basic examples and the necessary background on algebraic entropy, which is used in the rest of the paper.

\medskip
Denote by $\p$ the set of all prime numbers plus the symbol $\infty$. For every prime $p$ we denote by $\Q_p$ the field of $p$-adic numbers and by $|-|_p$ the $p$-adic norm on $\Q_p$; moreover, we let $\Q_\infty=\R$ and $|-|_\infty$ the usual absolute value on $\R$.
If $K_p$ is a finite extension of $\Q_p$, then we denote still by $|-|_p$ the unique extension of the $p$-adic norm to $K_p$.

Let $N$ be a positive integer, $p\in\p$ and $\phi_p:\Q_p^N\to \Q_p^N$ an endomorphism. Since $\phi_p$ is continuous, $\phi_p$ is $\Q_p$-linear and so its action on $\Q_p^N$, with respect to the canonical base of $\Q_p^N$ over $\Q_p$, is represented by an $N\times N$ matrix $M_{\phi_p}$ with coefficients in $\Q_p$. We call characteristic polynomial and eigenvalues of $\phi_p$ the characteristic polynomial and the eigenvalues of $M_{\phi_p}$.

\smallskip
Section \ref{p-adic} is dedicated to the following formula, that gives the value of the algebraic entropy of an endomorphism of $\Q_p^N$ in terms of its eigenvalues. 

\medskip
\noindent {\bf Fact A.}{ \em Let $N$ be a positive integer, $p\in\p$ and $\phi_p:\Q_p^N\to \Q_p^N$ an endomorphism. Then
$$h_A(\phi_p)=\sum_{|\lambda_i^{(p)}|_p>1}\log|\lambda_i^{(p)}|_p,$$
where $\{\lambda_i^{(p)}:i=1,\ldots, N\}$ are the eigenvalues of $\phi_p$, contained in some finite extension $K_p$ of $\Q_p$.}

\medskip
Section \ref{proof} contains the heart of the proof of the Algebraic Yuzvinski Formula. Let $N$ be a positive integer and $\phi:\Q^N\to \Q^N$ an endomorphism. For every $p\in\p$, $\Q$ can be identified with a subfield of $\Q_p$ and so $\phi$ induces an endomorphism $\phi_p:\Q_p^N\to \Q_p^N$ just extending the scalars, that is,
$$\phi_p=\phi\otimes_{\Q}id_{\Q_p}.$$
Since the algebraic entropy of each $\phi_p$ can be computed using the above Fact A, the idea in Section \ref{proof} is to express the algebraic entropy of $\phi$ in terms of the algebraic entropy of the $\phi_p$, with $p$ ranging in $\p$:

\medskip
\noindent {\bf Fact B. }{\em  Let $N$ be a positive integer, $\phi:\Q^N\to \Q^N$ an endomorphism and $\phi_p=\phi\otimes_{\Q}id_{\Q_p}$ for every $p\in\p$. Then $$h_A(\phi)=\sum_{p\in\p}h_A(\phi_p).$$}
\medskip

The main idea for the proof of the Algebraic Yuzvinski Formula relies in this step, that marks also the main difference between our approach to the Algebraic Yuzvinski Formula and the proof of the classical Yuzvinski Formula given by Lind and Ward in \cite{LW} (see Section \ref{LW-sub} for a detailed explanation).

\medskip
In Section \ref{conc} we consider the following known decomposition of the Mahler measure, similar to that obtained in Fact B for the algebraic entropy.

\medskip
\noindent{\bf Fact C. }{\em Let $N$ be a positive integer and $f(X)=sX^N+a_1 X^{N-1}+\ldots+a_N\in\Z[X]$ a primitive polynomial of degree $N$. For every $p\in\p$ let $\{\lambda^{(p)}_{i}:i=1,\ldots,N\}$ be the roots of $f(X)$, considered as an element of $\Q_p[X]$, in some finite extension $K_p$ of $\Q_p$.
For every prime $p$, $$\log|1/s|_p=\sum_{|\lambda^{(p)}_i|_p>1}\log|\lambda^{(p)}_i|_p.$$
Consequently, $$\log|s|=\sum_{p\in\p\setminus\{\infty\}} \sum_{|\lambda^{(p)}_i|_p>1}\log|\lambda^{(p)}_i|_p;$$
hence
$$m(f(X))= \sum_{p\in\p}\sum_{|\lambda^{(p)}_i|_p>1}\log|\lambda^{(p)}_i|_p.$$
}

\medskip
Fact C explains the meaning of the ``mysterious" term $\log|s|$ appearing in the definition of the Mahler measure of a primitive polynomial $f(X)=sX^N+a_1 X^{N-1}+\ldots+a_N\in\Z[X]$. It can be deduced from the main results of \cite{LW}, but we include a direct proof for reader's convenience.

\medskip
Facts A, B and C together give the Algebraic Yuzvinski Formula. In particular, as a consequence of Fact C and Fact A, we have the following

\medskip
\noindent{\bf Corollary. }{\em Let $N$ be a positive integer, $\phi:\Q^N\to \Q^N$ an endomorphism and $p_\phi(X)=sX^N+a_1 X^{N-1}+\ldots+a_N$ the characteristic polynomial of $\phi$ over $\Z$. Let $p$ be a prime and $\phi_p=\phi\otimes_\Q id_{\Q_p}$.
Then $$h_A(\phi_p)=\log|1/s|_p.$$}

\medskip
In the final Section \ref{Appl} we give several applications of the Algebraic Yuzvinski Formula and discuss some open problems. 
In particular, we describe the fundamental results from \cite{DG} and \cite{DG1} about the algebraic entropy of endomorphisms of discrete Abelian groups, where the Algebraic Yuzvinski Formula applies. 
As the proof of the Algebraic Yuzvinski Formula is now self-contained and does not depend on its topological counterpart, also the treatment of the algebraic entropy in \cite{DG} and \cite{DG1} becomes independent from the results on topological entropy.  On the other hand, it is again the Algebraic Yuzvinski Formula that allows one to prove the fundamental connection between the algebraic entropy of endomorphisms of discrete Abelian groups and the topological entropy of endomorphisms of compact Abelian groups, namely, a general version of the Bridge Theorem proved in \cite{DG-BT}.

Finally, we give some open problems related to the applications, with the aim to understand whether and how the fundamental results for the algebraic entropy of endomorphisms of discrete Abelian groups can be extended to the general case of endomorphisms of LCA groups.

\subsection*{Aknowledgements}

We warmly thank Professor Dikranjan for reading at least two preliminary versions of the present paper and for his very useful comments and suggestions.
His encouragement is surely one of the ingredients of our proof of the Algebraic Yuzvinski Formula.

We are very grateful to Phu Chug for his careful reading of our manuscript and also for pointing out a mistake in an earlier version of this paper.

\section{Background on algebraic entropy}\label{Background}

In this section we recall some of the basic properties of the algebraic entropy proved in \cite{V}. These are useful tools in the computation of the algebraic entropy that we apply in the following sections. Note that in the discrete case one can prove stronger properties than in the general case -- see \cite{DG} for a complete study of the algebraic entropy of endomorphisms of discrete Abelian groups.

\subsection{Haar measure and modulus}  

We start recalling an easy result showing in particular that the value of the algebraic entropy does not depend on the choice of the Haar measure.  We use this fact each time we need to choose a Haar measure on a LCA group.

\begin{claim}\label{haar}{\em \cite[Lemma 2.1]{V}}
Let $G$ be an LCA group and $\mu$ a Haar measure on $G$.  If $\{A_n:n\in\N\}$ is a family of measurable subsets of $G$, then the quantity $l=\limsup_{n\to \infty}\frac{\log\mu(A_n)}{n}$ does not depend on the choice of $\mu$.
\end{claim}

The proof of this claim is a direct consequence of the fact that two different Haar measures on an LCA group are one multiple of the other.

\smallskip
The following are the first and fundamental examples. In particular, item (a) follows directly from Lemma \ref{monotonia}(1) below.

\begin{example}\label{ent=h_A}
\begin{enumerate}[\rm (a)]
\item If $G$ is a compact Abelian group, then $h_A(\phi)=0$ for every $\phi\in\End(G)$.
\item If $G$ is a discrete Abelian group, we can choose $\mu$ to be the cardinality of the subsets of $G$. With this choice, 
for endomorphisms of discrete Abelian groups, {\em our definition of $h_A(-)$ is exactly the definition of algebraic entropy given in \cite{DG}}.
\end{enumerate}
\end{example}

Denote by $\R_+$ the multiplicative group of positive reals. Fixed an LCA group $G$ and a Haar measure $\mu$ on $G$, the {\em modulus} is a group homomorphism 
$$\mod_G:\Aut(G)\to\R_+ , \text{ such that }\ \ \ \mu(\alpha E)=\mod_G(\alpha)\mu(E)$$ for every $\alpha\in \Aut(G)$ and every measurable subset $E$ of $G$ (see  \cite[(15.26) pag. 208]{HR} for the proof of the existence of the modulus). 

\medskip
The first three examples below are well-known, for the last two we refer to \cite[Chapter 1, \S2]{We}.

\begin{example}\label{example}
\begin{enumerate}[\rm (a)]
\item If $\alpha:\Z^N\to \Z^N$ is an automorphism, then $\mod_{\Z^N}(\alpha)=1$. More generally, $\mod_G \equiv 1$ if $G$ is a compact or discrete Abelian group.
\item If $\alpha:\R^N\to \R^N$ is an  automorphism, then $\mod_{\R^N}(\alpha)=|\det(\alpha)|$.
\item If $p$ is a prime and $\alpha:\Q_p^N\to \Q_p^N$ an automorphism, then $\mod_{\Q_p^N}(\alpha)=|\det(\alpha)|_p$.
\item If $\alpha:\C^N\to \C^N$ is a $\C$-linear automorphism, then $\mod_{\C^N}(\alpha)=|\det(\alpha)|^2$.
\item If $p$ is a prime, $K_p$ is a finite extension of $\Q_p$ of degree $d_p$ and $\alpha:K_p^N\to K_p^N$ is a $K_p$-linear automorphism, then $\mod_{K_p^N}(\alpha)=|\det(\alpha)|^{d_p}_p$.
\end{enumerate}
\end{example}

\subsection{Basic properties of algebraic entropy} \label{Bp-ae}

We start with the monotonicity property of $H_A(\phi,-)$ given in item \1 of the following lemma. Moreover, item \2 shows in particular that in order to compute the algebraic entropy of an endomorphism of an LCA group $G$, it suffices to consider a cofinal subfamily $\c'$ of $\c(G)$. This property is applied in crucial steps of the proof of the Algebraic Yuzvinski Formula. We recall that, given a poset $(S,\leq)$, a subset $T\subseteq S$ is said to be {\em cofinal} if, for every $s\in S$ there exists $t\in T$ such that $s\leq t$.

\begin{lemma}\label{monotonia}
Let $G$ be an LCA group and $\phi\in\End(G)$.
\begin{enumerate}[\rm (1)]
\item If $C,C'\in\c(G)$ and $C\subseteq C'$, then $H_A(\phi,C)\leq H_A(\phi,C')$.
\item If $\c_1\subseteq \c_2\subseteq \c(G)$ and $\c_1$ is  cofinal in $\c_2$, then 
$$\sup\{H_A(\phi,K):K\in\c_1\}=\sup\{H_A(\phi,K):K\in\c_2\}.$$
\item If $N$ is an open $\phi$-invariant subgroup of $G$, then $H_A(\phi\restriction_N,C)= H_A(\phi,C)$
for every $C\in\c(N)$. In particular, $h_A(\phi\restriction_N)\leq h_A(\phi)$.
\end{enumerate}
\end{lemma}
\begin{proof}
(1) comes directly from the definitions and (2) follows from (1). To prove (3), consider a Haar measure $\mu$ on $G$. Since $N$ is open in $G$, the restriction of $\mu$ to the Borel subsets of $N$ induces a Haar measure $\mu'$ on $N$. With this choice of the measures it is easy to see that $H_A(\phi\restriction_N,C)= H_A(\phi,C)$ for every $C\in\c(N)$.
\end{proof}

Item (3) of the above lemma shows that $h_A(-)$ is monotone under restriction to open invariant subgroups, and so it implies in particular the known monotonicity of $h_A(-)$ under restriction to invariant subgroups of discrete Abelian groups (see also Problem \ref{pb-bp}). 

\medskip
As a corollary of Lemma \ref{monotonia}(2) we see that the algebraic entropy $h_A(-)$ coincides with the algebraic entropy $\ent(-)$ introduced by Weiss on endomorphisms of discrete Abelian groups. 

\begin{corollary}\label{hA=ent}
Let $G$ be a torsion discrete Abelian group and $\phi\in\End(G)$. Then $h_A(\phi)=\ent(\phi)$.
\end{corollary}
\begin{proof}
For every compact (i.e., finite) $C\in\c(G)$, the subgroup $\langle C\rangle$ generated by $C$ is still finite. Hence the family of finite subgroups of $G$ is cofinal in $\c(G)$. By Claim \ref{haar} we can choose on $G$ the Haar measure given by the cardinality of subsets. We can now apply Lemma \ref{monotonia}(2) to obtain that
$$h_A(\phi)=\sup\left\{\lim_{n\to \infty}\frac{\log|T_n(\phi,C)|}{n}:C\text{ finite subgroup of $G$}\right\},$$
which is exactly the definition of $\ent(\phi)$.
\end{proof}

Let  $G$ be an LCA group, $\mu$ a Haar measure on $G$, $\phi\in\End(G)$ and $C\in \c(G)$. If the sequence $\left\{\frac{\log\mu(T_n(\phi,C))}{n}:n\in\N_+\right\}$ is convergent, we say that {\em the $\phi$-trajectory of $C$ converges}. In particular, if the $\phi$-trajectory of $C$ converges, then the $\limsup$ in the definition of $H_A(\phi,C)$ becomes a limit:
$$H_A(\phi,C)=\lim_{n\to\infty}\frac{\log\mu(T_n(\phi,C))}{n}.$$

\begin{example}\label{h=h}
If $G$ is a compact or discrete Abelian group, then the $\phi$-trajectory of $C$ converges for every 
$\phi\in\End(G)$ and $C\in\c(G)$. Indeed, the compact case is obvious as the values of the measure form a bounded subset of the reals (so the above sequence always converges to $0$). For the discrete case we refer to \cite[Corollary 2.2]{DG}.
\end{example}

The following proposition collects some properties proved in \cite[Proposition 2.7, Corollary 2.9]{V}. We remark that item (2) is stated here in a slightly stronger form. We do not give its proof, as it is analogous to the one of the forthcoming Proposition \ref{prop'}(2).

For $G$, $G'$ LCA groups and $\phi\in\End(G)$, $\phi'\in\End(G')$ we say that $\phi$ and $\phi'$ are \emph{conjugated} by a topological isomorphism $\alpha:G\to G'$ if $\phi=\alpha^{-1}\phi'\alpha$.

\begin{proposition}\label{prop}
Let $G$ and $G'$ be LCA groups, $\phi\in\End(G)$ and $\phi'\in\End(G')$. 
\begin{enumerate}[\rm (1)] 
\item If $\phi$ and $\phi'$ are conjugated by a topological isomorphism $\alpha : G \to G'$, then 
$H_A(\phi,C)=H_A(\phi',\alpha C)$
for every $C\in\c(G)$. In particular, $h_A(\phi) = h_A(\phi')$. 
\item For every $C\in\c(G)$ and $C'\in \c(G')$,
\begin{equation}\label{ATPP}
H_A(\phi\times\phi',C\times C')\leq H_A(\phi,C)+H_A(\phi',C').
\end{equation}
In particular, $h_A(\phi \times \phi')\leq h_A(\phi)+h_A(\phi')$. Furthermore, if the $\phi$-trajectory of $C$ converges, then equality holds in \eqref{ATPP}.
 \item Let $\Phi=\phi\times\phi:G\times G\to G\times G$ and $C\in\c(G)$. Then 
\begin{equation*}H_A(\Phi,C\times C)= 2H_A(\phi,C) .\end{equation*}
In particular, $h_A(\Phi)=2 h_A(\phi).$
\item If $\phi\in\Aut(G)$, then $h_A(\phi^{-1} ) = h_A(\phi)-\log (\mod_G(\phi))$. 
\end{enumerate}
\end{proposition}

\begin{remark}\label{h_infty=h_A}
Let $G$ be an LCA group and $\phi\in\Aut(G)$. As we mentioned in the Introduction, Peters' algebraic entropy of $\phi$ is $h_\infty(\phi)=h_A(\phi^{-1})$ by definition. Hence, by Proposition \ref{prop}(4) we obtain that $$h_\infty(\phi)=h_A(\phi)-\log(\mod_G(\phi)).$$ In particular, in view of Example \ref{example}(a), $h_\infty(\phi)=h_A(\phi)$ whenever $G$ is discrete or compact.
\end{remark}

The following lemma shows the continuity of the algebraic entropy for direct limits of open invariant subgroups.

\begin{lemma}\label{upp-cont}
Let $G$ be an LCA group, $\phi\in \End(G)$, and suppose $\{ N_i  :  i \in I\}$ to be a directed system of open $\phi$-invariant subgroups of $G$ such that $G=\varinjlim N_i$. Then $h_A(\phi)=\sup_{i\in I}h_A(\phi\restriction_{N_i})$.
\end{lemma}
\begin{proof}
By Lemma \ref{monotonia}(3), we have that $h_A(\phi)\geq \sup_{i\in I}h_A(\phi\restriction_{N_i})$. On the other hand, consider $K\in\c(G)$. Then $K=\bigcup_{i\in I}(K\cap N_i)$ and so, by compactness, there exists a finite subset $F\subseteq I$ such that $K=\bigcup_{i\in F}(K\cap N_i)$. Furthermore, being $\{ N_i  :  i \in I\}$ directed, there exists $N\in \{ N_i  :  i \in I\}$ such that $\sum_{i\in F}N_i\subseteq N$ and so $K=\bigcup_{i\in I}(K\cap N_i)\subseteq K\cap N\subseteq N$. To conclude, notice that
$$H_A(\phi,K)= H_A(\phi\restriction_{N}, K)\leq h_A(\phi\restriction_N)\leq \sup_{i\in I}h_A(\phi\restriction_{N_i}),$$
where the first equality follows by Lemma \ref{monotonia}(3).
\end{proof}

We conclude this section with an example of computation of the algebraic entropy that is used later on. Note that it can be deduced from \cite[Example 1.9]{DGSZ} and it is proved in a slightly different way in \cite[Example 2.10]{DG}. 

\begin{example}\label{bernoulli}
Let $K$ be a discrete Abelian group and $G=\bigoplus_{n\in\N}K_n$, where each $K_n=K$. The {\em right Bernoulli shift} is the endomorphism of $G$ defined by
$$\beta_{K}:G\to G\  \text{such that}\  (x_0,x_1,\ldots,x_n,\ldots)\mapsto (0,x_0,\ldots,x_n,\ldots).$$ 
Then $h_A(\beta_K)=\log|K|$, with the usual convention that $\log|K|=\infty$ if $K$ is infinite.

Indeed, fix on $G$ the Haar measure given by the cardinality of subsets, and let $F\in\c(K_0)$. An easy computation shows that $|T_n(\beta_K,F)|=|F\times\beta_K(F)\times\ldots\times \beta^{n-1}_K(F)|=|F|^{n}$, hence $H_A(\beta_K,F)=\log|F|$, and so
$$h_A(\beta_K)\geq \sup\{H_A(\beta_K,F):F\in\c(K_0)\}=\sup\{\log|F|:F\in\c(K_0)\}=\log|K|.$$
If $K$ is infinite, then $\log|K|=\infty$ and the proof is concluded. So assume $|K|$ to be a positive integer.
The family of subgroups of the form $\overline{K}_i=K_0\oplus\ldots\oplus K_i$, with  $i\in\N$, is cofinal in $\c(G)$ and Lemma \ref{monotonia}\2 gives
$$h_A(\beta_K)=\sup\{H_A(\beta_K,\overline{K}_i):i\in\N\}.$$
Now $T_n(\beta_K,\overline K_i)=K_0\oplus \ldots\oplus K_{i+n-1}$. Therefore, $|T_n(\beta_K,\overline K_i)|=|K|^{i+n}$ and so
$$H_A(\beta_i,\overline K_i)=\lim_{n\to\infty}\frac{\log|K|^{i+n}}{n}=\log|K|.$$
\end{example}

\subsection{The minor trajectory} 

We now recall a technique from \cite{V}, partially modifying it, that allows us to find convenient lower bounds for the algebraic entropy. 
Let $G$ be an LCA group, $\mu$ a Haar measure on $G$ and $\phi\in\End(G)$. For every $C\in\c(G)$,
$$T_n^{\leq}(\phi,C)=C+\phi^{n-1}C$$
is the {\em minor $n$-th $\phi$-trajectory of $C$}. Furthermore, let $$H^{\leq}(\phi,C)=\limsup_{n\to \infty}\frac{\log\mu(T^{\leq}(\phi,C))}{n}.$$
In view of Claim \ref{haar}, the value of $H^{\le}(\phi,C)$ does not depend on the choice of $\mu$. 

\smallskip
The following example shows that the minor trajectory is of no help in the discrete case. 

\begin{example}
Let $G$ be a discrete Abelian group, $\phi\in\End(G)$ and $C\in\c(G)$. Then 
$$H^{\leq}(\phi,C)=\limsup_{n\to \infty}\frac{\log|C+\phi^{n-1}C|}{n}\leq \limsup_{n\to \infty}\frac{\log|C|+\log|\phi^{n-1}C|}{n}\leq \limsup_{n\to \infty}\frac{2\log|C|}{n}=0 .$$
This shows that $H^\leq(\phi,-)\equiv 0$ in the discrete case.
\end{example}

If the sequence $\left\{\frac{\log\mu(T_n^{\leq}(\phi,C))}{n}:n\in\N_+\right\}$ converges, then we say that \emph{the minor $\phi$-trajectory of $C$ converges}.

\smallskip
Since our definitions are a modification of those in \cite{V}, we give a proof of the following two results, which are the counterparts of \cite[Lemma 2.10]{V} and \cite[Proposition 2.11]{V} respectively.

\begin{lemma}\label{base-leq}
Let $G$ be an LCA group, $\phi\in\End(G)$ and $C\in\c(G)$. Then:
\begin{enumerate}[\rm (1)]
\item $0\leq H^{\leq}(\phi,C)\leq H_A(\phi,C)\leq h_A(\phi)$;
\item $\log(\mod_G(\phi))\leq H^{\leq}(\phi,C)$, if $\phi$ is an automorphism.
\end{enumerate}
\end{lemma}
\begin{proof}
Let $\mu$ be a Haar measure on $G$.

\smallskip
\1 Let $n\in\N_+$ and note that $T_n^\leq(\phi,C)\subseteq T_n(\phi,C)$; using this inclusion and the monotonicity of $\mu$, it is not difficult to show that $H^{\leq}(\phi,C)\leq H_A(\phi,C)$. All the other inequalities in the statement are a direct consequence of the definitions. 

\smallskip
\2 Assume that $\phi\in\Aut(G)$ and let $n\in\N_+$. Since $\phi^nC\subseteq C+\phi^n C$, we get
\begin{equation*}\begin{split}
\log(\mod_G(\phi))=\lim_{n\to\infty}\frac{\log(\mod_G(\phi)^{n-1}\mu(C))}{n}=\lim_{n\to\infty}\frac{\log\mu(\phi^{n-1}C)}{n}\\
\leq\limsup_{n\to\infty}\frac{\log\mu(C+\phi^{n-1}C)}{n}=H^\leq(\phi,C) .
\end{split}\end{equation*}
This concludes the proof.
\end{proof}

\begin{proposition}\label{prop'}
Let $G$ and $G'$ be LCA groups, $\phi\in\End(G)$ and $\phi'\in\End(G')$.
\begin{enumerate}[\rm (1)]
\item  If $\phi$ and $\phi'$ are conjugated by a topological isomorphism $\alpha : G \to G'$, then $H^\leq(\phi,C)=H^\leq(\phi',\alpha C)$ for every $C\in\c(G)$.
\item For every $C\in\c(G)$ and $C'\in\c(G')$,
\begin{equation}\label{ATPPP}
H^{\leq}(\phi\times\phi',C\times C')\leq H^{\leq}(\phi,C)+H^{\leq}(\phi',C') .
\end{equation}
If the minor $\phi$-trajectory of $C$ converges, then equality holds in \eqref{ATPPP}.

\item Let $\Phi=\phi\times\phi:G\times G\to G\times G$ and $C\in\c(G)$. Then 
\begin{equation*}
H^\leq(\Phi,C\times C)= 2 H^\leq(\phi,C).
\end{equation*}
\end{enumerate}
\end{proposition}
\begin{proof}
\1 Let $\mu$ be a Haar measure on $G$. For every Borel subset $E\subseteq G'$, let $\mu'(E)=\mu(\alpha^{-1}E)$. Then $\mu'$ is a Haar measure on $G'$. For $C\in\c(G)$ and $n\in\N_+$, $$\mu(T_n^\leq(\phi,C))=\mu'(\alpha T_n^\leq(\phi,C))=\mu'(T_n^\leq(\phi',\alpha C)) ;$$
consequently, $H^\leq(\phi,C)=H^\leq(\phi',\alpha C)$.

\smallskip
\2 Let $\mu$ and $\mu'$ be Haar measures on $G$ and $G'$ respectively. It is known that there exists a Haar measure  $\mu\times \mu'$ on $G\times G'$ such that $(\mu\times\mu')(E\times E')=\mu(E)\mu'(E')$ for every measurable $E\subseteq G$, $E'\subseteq G'$. Let now $C\in\c(G)$, $C'\in\c(G')$ and $n\in\N_+$. Then $T_n^\leq(\phi\times\phi',C\times C')=T_n^\leq(\phi,C)\times T_n^\leq(\phi',C')$, and so
\begin{align}\label{leq->=}
\notag H^\leq(\phi\times\phi',C\times C')&=\limsup_{n\to\infty}\frac{\log(\mu\times\mu')(T_n^\leq(\phi\times\phi',C\times C'))}{n}\\
&\leq\limsup_{n\to\infty}\frac{\log\mu(T_n^\leq(\phi,C))}{n}+\limsup_{n\to\infty}\frac{\log\mu'(T_n^\leq(\phi',C'))}{n}\\
\notag &=H^\leq(\phi,C)+H^\leq(\phi',C') .
\end{align}
If the minor $\phi$-trajectory of $C$ converges, then equality holds in \eqref{leq->=}.

\smallskip
\3 If $G=G'$, $\phi=\phi'$ and $C=C'$, then again equality holds in \eqref{leq->=}.
\end{proof}

\section{Algebraic entropy of endomorphisms of $\Q_p^N$}\label{p-adic}

We start this section fixing some notations. 
Recall that we denote by $\p$ the set of all prime numbers plus the symbol $\infty$. All along this section $p$ denotes an arbitrarily fixed element of $\p$ and $N$ a fixed positive integer. 

When $p<\infty$, we denote by $\Q_p$ the field of $p$-adic numbers, which is the field of quotients of the ring $\Z_p$ of $p$-adic integers, that is, $\Q_p=\bigcup_{n\in\N}\frac{1}{p^n}\Z_p$. An arbitrary element $x$ of $\Q_p$ has a unique {\em $p$-adic expansion} of the form
$$x=x_{-n}p^{-n}+x_{-n+1}p^{-n+1}+\ldots+x_0+x_1 p+\ldots+x_k p^k+\ldots$$
for some $n\in\N$, and $0\leq x_i\leq p-1$ for every $i\geq -n$; for $x\neq 0$ we always assume that $x_{-n}\neq0$.
The {\em $p$-adic norm} of $x$ is
$$|x|_p=\begin{cases}p^n& \text{if $x\neq 0$;}\\
0 & \text{if $x=0$.}\end{cases}$$
When $p=\infty$, we let $\Q_\infty=\R$ and $|-|_\infty$ be the usual absolute value on $\R$.

\smallskip
For every $p\in\p$ and $\varepsilon \in \R_+$, we denote by
$$D_p(\varepsilon)=\{x\in\Q_p : |x|_p\leq \varepsilon\}$$
the {\em disc} in $\Q_p$ of radius $\varepsilon$ centered at $0$. The family $\F_p=\{D_p(\varepsilon):\varepsilon \in\R_+\}$ is a base of compact neighborhoods of $0$ in $\Q_p$.

\begin{remark}\label{rm}
If $p$ is finite, $D_p(1)=\Z_p$ is the ring of $p$-adic integers. More generally, since for every non-trivial $x\in\Q_p$, $|x|_p=p^m$ for some $m\in\Z$, we have that $D_p(\varepsilon)=D_p(p^m)$, where $m$ is the largest integer such that $p^m\leq\varepsilon$, and $D_p(p^m)=p^{-m}\Z_p$ for every $m\in\Z$.

Thus $\mathcal F_p=\{p^{-m}\Z_p:m\in\Z\}$, where $p^{m}\Z_p\subseteq p^{m-1}\Z_p$ for every $m\in\Z$.
\end{remark}

\subsection{Finite extensions $K_p$ of $\Q_p$}

All along this section, $K_p$ denotes a finite extension of $\Q_p$ of degree
$$d_p=[K_p:\Q_p].$$ 
We denote again by $|-|_p$ the unique extension of the $p$-adic norm to $K_p$. Note that $K_\infty$ can be only either the trivial extension $K_\infty=\R$ or $K_\infty=\C$. In the first case $d_\infty=1$, while in the second case $d_\infty=2$ and $|-|_\infty$ is the usual norm on $\C$.

Fix on $K_p^N$ the $\max$-norm with respect to the canonical base, that is,
for every $x=(x_i)_{i=1}^N\in K_p^N$, let
$$|x|_p=\max\{|x_i|_p:i=1,\ldots,N\}.$$
For $\varepsilon\in\R_+$, the disc in $K_p^N$ centered at $0$ of radius $\varepsilon$ is
$$D_{p,N}(\varepsilon)=\{x\in K_p^N: |x|_p\leq \varepsilon\};$$ sometimes we denote $D_{p,1}(\varepsilon)$ simply by $D_p(\varepsilon)$.

The family $\{D_{p,N}(\varepsilon):\varepsilon \in\R_+\}$ is a base of compact neighborhoods of $0$ in $K_p^N$. 

\begin{remark}\label{rm1}
For every $i\in\{1,\ldots,N\}$ denote by $\pi_i:K_p^N\to K_p$ the $i$-th natural projection. A subset $S$ of $K_p^N$ is said to be {\em rectangular} if it coincides with the cartesian product $\pi_1(S)\times\ldots\times \pi_N(S)$.  
For the choice of the max-norm, $D_{p,N}(\varepsilon)$ is rectangular and all the projections $\pi_i(D_{p,N}(\varepsilon))$ coincide with $D_{p,1}(\varepsilon)$, that is $D_{p,N}(\varepsilon)=D_{p,1}(\varepsilon)^N$. 

Let $\mu $ and $\mu_N$ be the unique Haar measures on $K_p$ and on $K_p^N$ respectively, such that $\mu(D_{p}(1))=1$ and $\mu_N(D_{p,N}(1))=1$. By the uniqueness of the Haar measure, $\mu_N$ is the product measure of the measures $\mu$ taken on each copy of $K_p$. In particular,
\begin{equation}\label{DpN}
\mu_N(D_{p,N}(\varepsilon))=\mu_N(D_{p,1}(\varepsilon)^N)=\mu(D_{p}(\varepsilon))^N.
\end{equation}
\end{remark}

\smallskip
In the following lemma we use Remarks \ref{rm} and \ref{rm1} to estimate the measure of the discs $D_{p,N}(\varepsilon)$ in $K_p^N$.

\begin{lemma}\label{misurando_rettangoli}
Let $\varepsilon\in\R_+$, $p\in\p$ and let $\mu_p$ be the unique Haar measure on $K_p^N$ such that $\mu_p(D_{p,N}(1))=1$. Then $\mu_p(D_{p,N}(\varepsilon))\leq \varepsilon^{d_pN}$.
\end{lemma}
\begin{proof}
By \eqref{DpN} in Remark \ref{rm1} we can assume $N=1$. We divide the proof in two cases.

First suppose $p=\infty$. Then $D_\infty(\varepsilon)=\varepsilon D_\infty(1)$. The scalar multiplication by $\varepsilon$ is an automorphism of $K_\infty$,  that we denote by $\varphi_\varepsilon$. Then, by the definition of modulus and Example \ref{example}(b,d), we get $$\mu_\infty(D_\infty(\varepsilon))=\mu_\infty(\varepsilon D_\infty(1))=|\det {}_{K_\infty}(\varphi_\varepsilon)|^{d_\infty}\mu_\infty(D_\infty(1))=\varepsilon^{d_\infty} .$$

Suppose now that $p$ is finite. By Remark \ref{rm} there exists $m\in\Z$ such that $p^m\leq \varepsilon$ and $D_p(\varepsilon)=D_p(p^m)=p^{-m}D_p(1)$. The scalar multiplication by $p^{-m}$ is an automorphism of $K_p$,  that we denote by $\varphi_{p^{-m}}$. Then, using the definition of modulus and Example \ref{example}(e), we get
 $$\mu_p(D_p(\varepsilon))=\mu_p(D_p(p^m))=\mu_p(p^{-m}D_p(1))=|\det {}_{K_p}(\varphi_{p^{-m}})|_p^{d_p}\mu_p(D_p(1))=p^{md_p}\leq \varepsilon^{d_p},$$
 as desired.
 \end{proof}

\smallskip
For a $K_p$-linear endomorphism $\phi:K_p^N\to K_p^N$, we define the {\em norm of $\phi$} as
\begin{equation}\label{nnnn}
||\phi||_p=
\max\left\{\sum_{j=1}^N|a_{ij}|_p:i=1,\ldots,N\right\},
\end{equation}
where $M_\phi=(a_{ij})_{i,j}$ is the $N\times N$ matrix associated to $\phi$ with respect to the canonical base of $K_p^N$. 
It is well-known (and easily verified) that $|\phi(x)|_p\leq||\phi||_p|x|_p$ for every $ x\in K_p^N.$ Equivalently,
\begin{equation}\label{imm-palle}
\phi(D_{p,N}(\varepsilon))\subseteq D_{p,N}(||\phi||_p \varepsilon) .
\end{equation}

\begin{remark}
For a {\em finite} $p\in\p$, the natural choice for the norm to consider in \eqref{nnnn} should be $$\max\left\{|a_{ij}|_p:1\leq i,j\leq N\right\} .$$ This would allow a better approximation in \eqref{imm-palle}. Nevertheless, we prefer the norm as defined in \eqref{nnnn} as it permits to treat the case $p=\infty$ together with the case when $p$ is finite.
\end{remark}

\subsection{Algebraic entropy in $K_p^N$, when $K_p$ contains the eigenvalues}

All along this and the following subsection we fix $p\in\mathbb P$. Let $\lambda\in K_p$. An $N\times N$ matrix $J$ with coefficients in $K_p$ is a {\em Jordan block} relative to $\lambda$ if all the entries on the diagonal of $J$ are equal to $\lambda$, all the entries on the first superdiagonal are equal to $1$ and all the other entries are equal to $0$:
$$J=\begin{pmatrix}\lambda & 1 & 0 & \ldots & 0 \\
0 & \lambda & 1 & \ldots & 0\\
\vdots &\ddots &\ddots & \ddots &\vdots\\
0 & \ldots & 0 & \lambda & 1\\
0 & \ldots & 0 & 0 &\lambda
\end{pmatrix}$$ 
It is well-known from linear algebra that  for $s\in\N_+$ the matrix $J^s$ is an upper triangular matrix with $\lambda^s$ on the diagonal and 
$\binom{s}{j}\lambda^{s-j}$
on the $j$-th superdiagonal, for $j=1,2,\ldots,\min\{s,N-1\}$; in case $s<N-1$, the values above the $s$-th superdiagonal are all zero.

An $N\times N$ matrix $M$ with coefficients in $K_p$ is said to be in {\em Jordan form} if it is a block matrix whose diagonal blocks are Jordan blocks  and all the other blocks are zero.

\medskip
In the following lemma and proposition we compute the algebraic entropy of a $K_p$-linear endomorphism of $K_p^N$ whose matrix is a single Jordan block. 

\begin{lemma}\label{uff-p}
Let $\phi:K_p^N\to K_p^N$ be a $K_p$-linear endomorphism whose matrix is a Jordan block relative to $\lambda\in K_p$, let $n\in\N_+$, and $\varepsilon\in\R_+$. 
\begin{enumerate}[\rm (1)]
\item If $|\lambda|_p\leq 1$, then $T_n(\phi,D_{p,N}(\varepsilon))\subseteq D_{p,N}(n^{N+1} N \varepsilon)$.
\item If $|\lambda|_p>1$, then $T_n(\phi,D_{p,N}(\varepsilon))\subseteq D_{p,N}(|\lambda|_p^{n} n^{N+1} N \varepsilon)$.
\end{enumerate}
\end{lemma}
\begin{proof}
For every $s\in\N$, the explicit form of the matrix gives
\begin{equation}\label{J<-p}
||\phi^s||_p=\max\left\{\sum_{j=0}^{\min\{s,N-i\}}\left|\binom{s}{j}\lambda^{s-j} \right|_p:i=1,\ldots,N\right\}  \leq s^N \sum_{j=0}^{\min\{s,N-1\}}|\lambda|_p^{s-j} .
\end{equation}

\smallskip
\1  If $|\lambda|_p\leq 1$, then \eqref{J<-p} gives $||\phi^s||_p\leq Ns^N$ for every $s\in\N$. Consequently, by \eqref{imm-palle}
$$T_n(\phi,D_{p,N}(\varepsilon))\subseteq D_{p,N}(\varepsilon)+D_{p,N}(N\varepsilon)+ \ldots + D_{p,N}((n-1)^NN\varepsilon) \subseteq D_{p,N}(n^{N+1}N\varepsilon) .$$

\smallskip
\2 If $|\lambda|_p\geq 1$, then \eqref{J<-p} gives $||\phi^s||_p\leq Ns^N |\lambda|_p^s$ for every $s\in\N$. So, by \eqref{imm-palle}
$$T_n(\phi,D_{p,N}(\varepsilon))\subseteq D_{p,N}(\varepsilon)+  D_{p,N}(|\lambda|_pN\varepsilon)+\ldots + D_{p,N}((n-1)^N |\lambda|_p^{n-1} N\varepsilon)\subseteq D_{p,N}(n^{N+1}|\lambda|_p^{n}N\varepsilon).$$
This concludes the proof.
\end{proof}

\begin{proposition}\label{jordan-p}
Let $\phi : K_p^N\to K_p^N$ be a $K_p$-linear endomorphism whose matrix is a Jordan block relative to $\lambda\in K_p$. Then, for every $\varepsilon \in \R_+$,
$$H^{\leq}(\phi, D_{p,N}(\varepsilon)) = H_A(\phi, D_{p,N}(\varepsilon)) = \begin{cases}0 & \text{if}\ |\lambda|_p\leq 1, \\ d_pN\cdot \log|\lambda|_p & \text{if}\ |\lambda|_p>1.\end{cases}
$$
Furthermore, the $\phi$-trajectory and the minor $\phi$-trajectory of $D_{p,N}(\varepsilon)$ converge.
\end{proposition}
\begin{proof}
Let $\mu_p$ be the unique Haar measure on $K_p^N$ such that $\mu_p(D_{p,N}(1))=1$. For every $\varepsilon\in\R_+$, by Lemma \ref{base-leq}(1),
\begin{equation}\label{0leq-p}
0\leq H^{\leq}(\phi, D_{p,N}(\varepsilon)) \leq H_A(\phi, D_{p,N}(\varepsilon)).
\end{equation} 

\smallskip
Suppose that $|\lambda|_p\leq 1$. By Lemma \ref{uff-p}\1, $T_n(\phi,D_{p,N}(\varepsilon))\subseteq D_{p,N}(n^{N+1}N\varepsilon)$, and so
\begin{equation*}\begin{split}
0\leq H_A(\phi,D_{p,N}(\varepsilon))= \limsup_{n\to\infty}\frac{\log\mu_p(T_n(\phi,D_{p,N}(\varepsilon)))}{n}\leq
\lim_{n\to\infty}\frac{\log\mu_p(D_{p,N}(n^{N+1}N\varepsilon))}{n}\leq\\
\overset{(*)}{\leq}\lim_{n\to\infty}\frac{\log ((n^{N+1}N\varepsilon)^{d_pN})}{n} =\lim_{n\to \infty}\frac{d_p N(N+1) \log n+d_p N \log (N\varepsilon)}{n}=0 ,
\end{split}\end{equation*}
where the inequality $(*)$ comes from Lemma \ref{misurando_rettangoli}.
The thesis now follows from \eqref{0leq-p}.

\smallskip
On the other hand, if $|\lambda|_p>1$, then $T_n(\phi,D_{p,N}(\varepsilon))\subseteq D_{p,N}(|\lambda|_p^{n} n^{N+1} N \varepsilon)$ by Lemma \ref{uff-p}\2, and so
\begin{equation}\begin{split}\label{leq1-p}
H_A(\phi,D_{p,N}(\varepsilon))= \limsup_{n\to\infty}\frac{\log\mu_p(T_n(\phi,D_{p,N}(\varepsilon)))}{n}= \limsup_{n\to\infty}\frac{\log\mu_p(D_{p,N}(|\lambda|_p^{n} n^{N+1} N \varepsilon))}{n}\leq \\
\overset{(**)}{\leq} \lim_{n\to\infty}\frac{\log((|\lambda|_p^{n} n^{N+1}N \varepsilon)^{d_pN})}{n}
=\lim_{n\to\infty}\frac{d_pNn\log|\lambda|_p+ d_pN\log(n^{N+1} N \varepsilon)}{n}= d_p N\cdot \log|\lambda|_p ,
\end{split}\end{equation}
where the inequality $(**)$ comes from Lemma \ref{misurando_rettangoli}. Furthermore, $\phi$ is an automorphism and so, by Example \ref{example}(e) and by Lemma \ref{base-leq}\2,
\begin{equation}\label{leq2-p}
d_p N\cdot \log|\lambda|_p=\log(\mod_{K_p^N}(\phi))\leq H^{\leq}(\phi, D_{p,N}(\varepsilon)).
\end{equation}
The two inequalities in \eqref{leq1-p} and \eqref{leq2-p}, together with that of \eqref{0leq-p}, give the desired conclusion.
\end{proof}

Using the above results we can give now a general formula to compute the algebraic entropy of a $K_p$-linear endomorphism of $K_p^N$ having all  eigenvalues in the base field $K_p$.

\begin{proposition}\label{Yuz-p}
Let $\phi:K_p^N\to K_p^N$ be an endomorphism, $C\in\c(K_p^N)$ and assume that the eigenvalues $\{\lambda_i:i=1,\ldots,N\}$ of $\phi$ are contained in $K_p$. Then
\begin{equation*}
h_A(\phi)=H_A(\phi,C)=H^\leq(\phi,C)=\sum_{|\lambda_i|_p>1}d_p\cdot \log|\lambda_i|_p .
\end{equation*}
\end{proposition}
\begin{proof}
Denote by $M_\phi$ the matrix associated to $\phi$. It is well-known from linear algebra that there exist an invertible matrix $M$ and a matrix $J$ such that $M_\phi=M^{-1}JM$, with $J$ in Jordan form. Denote by $\psi$ the endomorphism associated to $J$ and by $\alpha$ the automorphism associated to $M$. By Proposition \ref{prop'}\1 and Proposition \ref{prop}\1, we have 
\begin{equation}\label{J}
H^{\leq}(\phi,C)=H^{\leq}(\psi,\alpha C),\ H_A(\phi,C)=H_A(\psi, \alpha C)\ \text{and}\ h_A(\phi)=h_A(\psi).
\end{equation} 
Clearly, $\alpha C\in\c(K_p^N)$ and so we can fix $\delta,\varepsilon\in\R_+$ such that
$$D_{p,N}(\delta)\subseteq \alpha C\subseteq D_{p,N}(\varepsilon).$$ 
Now, $K_p^N$ is a direct product of $\psi$-invariant subspaces on which  $\psi$ acts as a single Jordan block. 
By the existence of limits in Proposition \ref{jordan-p}, we can apply Propositions \ref{prop}\2 and \ref{prop'}\2 to obtain
\begin{align*}
\sum_{|\lambda_i|_p>1}d_p\cdot \log|\lambda_i|_p&=H^\leq(\psi, D_{p,N}(\delta))\leq H^\leq(\psi,\alpha C)\leq \\ 
\notag &\leq H_A(\psi, \alpha C)\leq H_A(\psi, D_{p,N}(\varepsilon))=\sum_{|\lambda_i|_p>1}d_p\cdot \log|\lambda_i|_p.
\end{align*}
Since $\alpha$ induces a bijection of $\c(K_p^N)$ onto itself, in particular for every $C\in\c(K_p^N)$ we have
\begin{equation}\label{res1}
H^\leq(\psi,C)=H_A(\psi,C)=\sum_{|\lambda_i|_p>1}d_p\cdot \log|\lambda_i|_p.
\end{equation}
Consequently,
\begin{equation}\label{res2}
h_A(\psi)=\sum_{|\lambda_i|_p>1}d_p\cdot \log|\lambda_i|_p.
\end{equation} 
Now \eqref{J} applied to \eqref{res1} and \eqref{res2} gives the desired conclusion.
\end{proof}

\subsection{General formula for the algebraic entropy in $\Q_p^N$}

Applying Proposition \ref{Yuz-p}, in the following theorem we can compute the algebraic entropy of an endomorphism of $\Q_p^N$ in terms of the eigenvalues of its matrix. This is a more precise version of Fact A announced in the Introduction. As a consequence we improve Proposition \ref{Yuz-p} in Corollary \ref{coro2}.

\begin{theorem}\label{yuz-pp}
Let $\phi_p:\Q_p^N\to \Q_p^N$ be an endomorphism and $C\in\c(\Q_p^N)$. Then
$$h_A(\phi_p)=H_A(\phi_p,C)=H^{\leq}(\phi_p,C)=\sum_{|\lambda_i|_p>1}\log|\lambda_i|_p,$$
where $\{\lambda_i:i=1,\ldots,N\}$ are the eigenvalues of $\phi_p$ in some finite extension $K_p$ of $\Q_p$.
\end{theorem}
\begin{proof}
Let $d_p=[K_p:\Q_p]$. 
Extend $\phi_p$ to a $K_p$-linear endomorphism $\phi_{K_p}$ of $K_p^N$ simply by letting 
$$\phi_{K_p}=\phi_p\otimes_{\Q_p} id_{K_p};$$
the eigenvalues $\{\lambda_i:i=1,\ldots,N\}\subseteq K_p$ of $\phi_{K_p}$ are exactly the eigenvalues of $\phi_p$, since $\phi_{K_p}$ and $\phi_p$ are represented by the same matrix.

Fix a base $\{e_i:i=1,\ldots,d_p\}$ of $K_p$ over $\Q_p$. Then every $x\in K_p$ has coordinates $(x^{(1)},\ldots,x^{(d_p)})$ with respect to this base. Moreover, $K_p^N\cong (\Q_p^N)^{d_p}$ and this isomorphism is given by $\alpha:K_p^N\to (\Q_p^N)^{d_p}$, defined by 
$$\alpha(x_i)_{i=1}^N=((x_i^{(1)})_{i=1}^N,\ldots,(x_i^{(d_p)})_{i=1}^N)).$$
For $$\Phi=\underbrace{\phi_p\times\ldots\times\phi_p}_{d_p}:(\Q_p^N)^{d_p}\to (\Q_p^N)^{d_p},$$
an easy computation shows that $$\phi_{K_p}=\alpha^{-1} \Phi \alpha.$$
Let $C'\in\c(K_p^{N})$; so $\alpha C'\in\c((\Q_p^N)^{d_p})$ as well.
By Propositions \ref{prop'}\1 and \ref{prop}\1, 
$$H^{\leq}(\phi_{K_p},\alpha C')=H^{\leq}(\Phi, C'),\ \ H_A(\phi_{K_p}, \alpha C')=H_A(\Phi, C')\ \ \text{and}\ \ h_A(\phi_{K_p})=h_A(\Phi).$$ 
These equalities and Proposition \ref{Yuz-p} yield
\begin{equation}\label{JJJ}
h_A(\Phi)=H^{\leq}(\Phi, C')=H_A(\Phi, C')=\sum_{|\lambda_i|_p>1}d_p\cdot \log|\lambda_i|_p . 
\end{equation}
Let now $C\in\c(\Q_p^N)$; then $C'=\underbrace{C\times\ldots\times C}_{d_p}\in\c((\Q_p^N)^{d_p})$.
Since $\Phi=\underbrace{\phi_p\times\ldots\times\phi_p}_{d_p}$, \eqref{JJJ} together with Propositions \ref{prop}\3 and \ref{prop'}\3, and an obvious inductive argument, gives
$$h_A(\phi_p)=H^{\leq}(\phi_p, C)=H_A(\phi_p, C)=\sum_{|\lambda_i|_p>1}\log|\lambda_i|_p,$$
as desired.
\end{proof}

Consider a $K_p$-linear endomorphism $\phi:K_p^N\to K_p^N$. In particular, $\phi$ is conjugated to an endomorphism $\psi$ of $\Q_p^{d_pN}$. Furthermore, the set of the eigenvalues of $\psi$ over $\Q_p$ is a disjoint union of $d_p$ many copies of the set of the eigenvalues of $\phi$ over $K_p$. Hence, a consequence of the above theorem is that in Proposition \ref{Yuz-p} it is superfluous to assume the eigenvalues of $\phi$ to lie in the base field $K_p$:

\begin{corollary}\label{coro2}
Let $\phi:K_p^N\to K_p^N$ be a $K_p$-linear endomorphism and $C\in\c(K_p^N)$. Then
$$h_A(\phi)=H_A(\phi,C)=H^{\leq}(\phi,C)=\sum_{|\lambda_i|_p>1}d_p\cdot\log|\lambda_i|_p ,$$
where $\{\lambda_i:i=1,\ldots,N\}$ are the eigenvalues of $\phi $ in some finite extension of $K_p$.
\end{corollary}

\section{Algebraic entropy of endomorphisms of $\Q^N$}\label{proof}

We fix all along this section a positive integer $N$ and an endomorphism $\phi:\Q^N\to \Q^N$.

\subsection{First reduction}

To evaluate the algebraic entropy of $\phi$ one has to consider the growth of the trajectories of {\em all} the finite subsets of $\Q^N$ containing $0$.
We introduce a smaller family of finite subsets of $\Q^N$, that suffices to compute the algebraic entropy of $\phi$, as proved in Proposition \ref{simp1}. 

Let $\{e_i:i=1,\ldots,N\}$ be the canonical base of $\Q^N$ over $\Q$. For every $m\in\N_+$, let
$$E_m=\left\{\sum_{i=1}^Nc_i e_i:c_i=0,\pm1/m,\pm 2/m,\ldots,\pm m/m\right\}.$$ 
The following lemma is an easy application of the definition.

\begin{lemma}
If $m,m'\in\N_+$ and $m'$ divides $m$, then $E_{m'}\subseteq E_m$.
\end{lemma}

Let now  $a\in\Q$. We denote by $\varphi_a:\Q^N\to \Q^N$ the multiplication by $a$, namely $\varphi_a(x)=a\cdot x$ for every $x\in\Q^N$. If $a\neq 0$, then $\varphi_a$ is an automorphism of $\Q^N$ and the diagram
\begin{equation*}
\xymatrix{
\Q^N\ar[r]^{\phi}\ar[d]_{\varphi_a}&\Q^N\ar[d]^{\varphi_a}\\
\Q^N\ar[r]^{\phi} &\Q^N
}
\end{equation*}
commutes, that is, 
\begin{equation}\label{square}
\varphi_a\phi\varphi_a^{-1}=\phi.
\end{equation}

\begin{proposition}\label{simp1}
In the above notations,
${h_A(\phi)=\sup\{H_A(\phi,E_m):m\in\N_+\}.}$
\end{proposition}
\begin{proof}
By definition $h_A(\phi)\geq \sup_{m\in\N_+}H_A(\phi,E_m)$. On the other hand, let $F\in\c(\Q^N)$. There exist $s,t \in\N_+$ such that $F$ is contained in a set of the form
$$S_{s,t}=\left\{\sum_{i=1}^N(a_i/b_i) e_i:a_i=0,\pm 1,\pm 2,\ldots, \pm t;\ b_i=1,2,\ldots,s\right\}.$$ Lemma \ref{monotonia}(1) yields $H_A(\phi, F)\leq H_A(\phi, S_{s,t})$. For $a=1/t$, we have $\varphi_{a}S_{s,t}=E_{s t}$. In view of  Proposition \ref{prop}\1 and \eqref{square}, we obtain $$H_A(\phi,S_{s,t})=H_A(\varphi_{a}\phi \varphi^{-1}_{a}, \varphi_{a}S_{s,t})=H_A(\phi,E_{s t}).$$ Hence, $h_A(\phi)\leq \sup_{m\in\N_+}H_A(\phi,E_m)$. 
\end{proof}

\subsection{Subrings of the rationals}

A non-zero rational number $x$ can be written uniquely in the form $x=a/b$ with $a\in \Z$, $b\in\N_+$ and $(a,b)=1$; so we assume every non-zero rational number to be in this form.

For every subset $\P$ of $\p$, let 
\begin{equation}\label{ZP}
\Z_{(\P)}=\Z\left[1,1/p:p\in\P\setminus\{\infty\}\right]
\end{equation}
be the subring of $\Q$ generated by $1$ and all the elements of the form $1/p$ with $p\in\P\setminus\{\infty\}$. Note that $\Z_{(\P)}$ contains $\Z$ for every choice of $\P$; in particular, $\Z_{(\P)}=\Z$ if $\P=\emptyset$ or $\P=\{\infty\}$. Furthermore, if $\P=\p$, then $\Z_{(\P)}=\Q$ and if $\P=\p\setminus \{p\}$ with $p<\infty$, then $\Z_{(\P)}$ is isomorphic to the localization of $\Z$ at the prime ideal $p\Z$.

By definition, a non-zero rational number $a/b$ belongs to $\Z_{(\P)}$ if and only if all the primes dividing $b$ belong to $\P$. This is expressed equivalently in item (1) of the following lemma in terms of the $p$-adic values (see \eqref{Z_P} in Section \ref{LW-sub} for another equivalent description of $\Z_{(\P)}$). Item (2) is another basic property of the $p$-adic values of elements of $\Z_{(\P)}$ for $p\in\P$.

\begin{lemma}\label{subring_basic}
Let $\P$ be a subset of $\p$.
\begin{enumerate}[\rm (1)]
\item If $x\in\Q$, then $x\in \Z_{(\P)}$ if and only if $|x|_p\leq1$ for every $p\in\p\setminus\P$.
\item If $x\in\Z_{(\P)}$, $x\neq0$ and $|x|_\infty<1$, then $\max\left\{|x|_p:p\in\P\setminus\{\infty\}\right\}> 1$.
\end{enumerate}
\end{lemma}
\begin{proof}
Item (1) follows directly from the definition. The hypotheses of item (2) imply that $x$ has non-trivial denominator and so there is some prime $p$ dividing it, that is $|x|_p>1$. By item (1), $p\in\P$.
\end{proof}

Now we go back to our usual setting, that is, $N$ is a fixed positive integer and $\phi:\Q^N\to \Q^N$ is an endomorphism.

\begin{definition}
Let $M_\phi=(a_{ij})_{i,j}$ be the $N\times N$ rational matrix associated to $\phi$ and let $m$ be a positive integer. The set
$\P(\phi,m)$ is the minimal subset of $\p$ containing $\infty$ and such that $\{1/m\}\cup \{a_{ij}:1\leq i,j\leq N\}\subseteq\Z_{(\P(\phi,m))}$.
Furthermore, let $\Pi(\phi,m)=\P(\phi,m)\setminus\{\infty\}$.

Finally, we let $\P(\phi)=\P(\phi,1)$ and $\Pi(\phi)=\Pi(\phi,1)$.
\end{definition}

In other words, $p\in\p$ belongs to $\P(\phi,m)$ if and only if either $p$ divides the denominator of some $a_{ij}$, for $1\leq i,j\leq N$, or
$p$ divides $m$, or $p=\infty$. So in particular $p\in\Pi(\phi)$ if and only if $||\phi||_p> 1$.

\medskip
The following proposition shows how the subrings $\Z_{(\P(\phi,m))}$ are related with the subsets $E_m$ of $\Q^N$ introduced in the previous subsection.

\begin{proposition}\label{traj-sep}
Let $m$ be a positive integer. Then:
\begin{enumerate}[\rm (1)]
\item $(\Z_{(\P(\phi,m))})^N$ is a $\phi$-invariant subgroup of $\Q^N$ containing $E_m$;
\item $T_n(\phi,E_m)\subseteq (\Z_{(\P(\phi,m))})^N$ for every $n\in\N_+$.
\end{enumerate}
\end{proposition}
\begin{proof}
(1) Let $x\in (\Z_{(\P(\phi,m))})^N$. The components of $\phi(x)$ are sums of components of $x$ multiplied by coefficients of the matrix of $\phi$. 
By the definition of $\P(\phi,m)$ and since $\Z_{(\P(\phi,m))}$ is a subring of $\Q$, we have that $\phi(x)\in(\Z_{(\P(\phi,m))})^N$.

\smallskip
(2) follows from (1).
\end{proof}

\subsection{From cardinality to measure}

We begin this subsection fixing some notation. 
For $p\in\p$, let $\alpha_p:\Q^N\to \Q_p^N$ be the diagonal map of the natural embedding of $\Q$ in $\Q_p$.
Moreover, let $\P$ be a fixed finite subset of $\p$ containing $\infty$. The finite product $\prod_{p\in\P}\Q_p^N$ is an LCA group. 
For every $\varepsilon\in \R_+$, we set
$$\D_\P(\varepsilon)=D_{\infty,N}(\varepsilon)\times \prod_{p\in\P, p<\infty}D_{p,N}(1)\ \subseteq \prod_{p\in\P}\Q_p^N.$$
Furthermore, we denote the diagonal map of the embeddings $\alpha_p:\Q^N\to\Q_p^N$ by
$$\alpha_\P=\prod_{p\in\P}\alpha_p:\Q^N\to \prod_{p\in\P}\Q_p^N.$$


%

In these terms, we give a useful consequence of Proposition \ref{traj-sep}(2):

\begin{corollary}\label{traj-p-nonsep}
Let $m,n\in\N_+$ and $p\in\p\setminus\P(\phi,m)$. Then $\alpha_p(T_n(\phi,E_m))\subseteq D_{p,N}(1)$.
\end{corollary}
\begin{proof}
If $x\in T_n(\phi,E_m)$, then $|x|_p\leq1$ by Proposition \ref{traj-sep}(2), that is, $\alpha_p(x)\in D_{p,N}(1)$.
\end{proof}

Moreover, one can state a slightly different interpretation of Lemma \ref{subring_basic}(2). In fact, given a subset $\P$ of $\p$ such that $\infty\in\P$, and two distinct elements $x,y\in\Z_{(\P)}$, at least one among the $p$-adic distances $|x-y|_p$ (with $p\in\P$) is ``large''. Roughly speaking, the diagonal embedding $\Q\to\prod_{p\in\P}\Q_p$ ``separates" $x$ and $y$. This fact 
is fundamental for the following result, that explains how we pass from a finite subset $\mathcal F$ of $\Q^N$ to a measurable subset of a finite product $\prod_{p\in\P}\Q_p^N$ whose Haar measure coincides with the size of $\mathcal F$. To this end we use a finite subset $\P$ of $\p$ containing $\infty$ and such that $\F\subseteq (\Z_{(\P)})^N$. When this result applies in the sequel, $\F$ is always an $n$-th $\phi$-trajectory and $\P$ is of the form $\P(\phi,m)$, for some endomorphism $\phi$ of $\Q^N$ and some positive integers $n,m$.

\begin{proposition}\label{coro-mis}
Let $\P$ be a finite subset of $\p$ containing $\infty$,  and $k$ an integer $\geq 3$.
\begin{enumerate}[\rm (1)]
\item If $x,y\in(\Z_{(\P)})^N$ and $x\neq y$, then
$$(\alpha_\P(x)+ \D_\P(1/k))\cap (\alpha_\P(y)+\D_\P(1/k))=\emptyset.$$
\item If $\mathcal F\subseteq (\Z_{(\P)})^N$ is finite, then 
$$\mu\left(\alpha_\P(\mathcal F)+\D_{\P}(1/k)\right)=|\mathcal F|,$$
where $\mu$ is the Haar measure on $\prod_{p\in\P}\Q_p^N$ such that $\mu(\D_\P(1/k))=1$.
\end{enumerate}
\end{proposition}
\begin{proof}
(1) Denote by $x_i$ and $y_i$ (with $i=1,\ldots,N$) the components of $x$ and $y$ in the canonical base of $\Q^N$ over $\Q$. If there exists $i\in\{1,\ldots,N\}$ such that $|x_i-y_i|_\infty\geq 1$, then 
$$(\alpha_\infty(x)+D_{\infty,N}(1/k))\cap (\alpha_{\infty}(y)+D_{\infty,N}(1/k))=\emptyset.$$ 
On the other hand, if $|x_i-y_i|_\infty< 1$ for every $i=1,\ldots,N$, we can fix $i\in\{1,\ldots,N\}$ and use Proposition \ref{subring_basic}(2) to find a finite $p$ in $\P$ such that $|x_i-y_i|_p>1$. Therefore
$$(\alpha_p(x)+D_{p,N}(1))\cap (\alpha_p(y)+D_{p,N}(1))=\emptyset.$$

\smallskip
(2) Let $\mathcal F=\{f_i:i=1,\ldots, h\}$ for some positive integer $h$. We can suppose $f_i\neq f_j$ whenever $1\leq i\neq j \leq h$. By item (1) we have that $\bigcup_{i=1}^{h}(\alpha_\P(f_i)+\D_{\P}(1/k))$ is a disjoint union. By the definition of Haar measure we obtain
$\mu(\bigcup_{i=1}^{h}\alpha_\P(f_i)+\D_{\P}(1/k))=h=|\mathcal F|$. 
\end{proof}

\subsection{The $p$-adic contributions to the algebraic entropy}

All along this subsection we fix $m\in\N_+$. 
For every finite $p\in\p$ and $n\in\N_+$ we write
$$T_n^{p}(\phi,E_m)=\alpha_p T_n(\phi,E_m)+D_{p,N}(1),$$
for $p=\infty$ and $k\in\N_+$ with $k\geq 3$
$$T_n^{\infty}(\phi,E_m,k)=\alpha_\infty T_n(\phi,E_m)+D_{\infty,N}(1/k),$$
and
$$T_n^*(\phi,E_m,k)=\alpha_{\P(\phi,m)} T_n(\phi,E_m)+\D_{\P(\phi,m)}(1/k).$$

\begin{definition}
Consider a finite $p\in\p$ and a Haar measure $\mu_p$ on $\Q_p^N$. The {\em $p$-adic contribution to the algebraic entropy of $\phi$ at $E_m$} is
$$H^{p}(\phi, E_m)=\limsup_{n\to\infty}\frac{\log\mu_{p}(T^{p}_n(\phi, E_m))}{n}.$$
Consider also a Haar measure $\mu_\infty$ on $\Q_\infty^N=\R^N$ (for example one can take the usual Lebesgue measure of $\R^N$). The {\em $\infty$-adic contribution to  the algebraic entropy of $\phi$ at $E_m$} is
$$H^{\infty}(\phi, E_m)=\limsup_{n\to\infty}\frac{\log\mu_{\infty}(T^{\infty}_n(\phi, E_m,k))}{n}.$$
\end{definition}

Let us start with the following easy observation

\begin{lemma}\label{contributii}
If $p\in\p\setminus \P(\phi)$, then $H^p(\phi,E_m)=0$.
\end{lemma}
\begin{proof}
Choose a prime $p\in\p\setminus \P(\phi)$, then there exists $h\in\N_+$ such that $\alpha_p E_m\subseteq D_{p,N}(h)$. By definition of $\P(\phi)$, we have that $||\phi||_p\leq 1$ and so, $|\phi(x)|_p\leq ||\phi||_p|x|_p\leq h$ for all $x\in E_m$. In particular, using the fact that $|-|_p$ is non-Archimedean, this implies that $T_n^p(\phi,E_m)=\alpha_p T_n(\phi, E_m) +D_{p,N}(1)\subseteq D_{p,N}(h)$. For a given Haar measure $\mu$ on $\Q_p^N$ we get
$$H^p(\phi,E_m)\leq \limsup_{n\to\infty}\frac{\log\mu(D_{p,N}(h))}{n}=0,$$
as desired. 
\end{proof}

\medskip
For a fixed $p\in\p$, as noted above it is possible to extend $\phi$ to a $\Q_p$-linear endomorphism $\phi_p$ of the $\Q_p$-vector space $\Q_p^N$ extending the scalars, that is, $$\phi_p=\phi\otimes_{\Q}id_{\Q_p}:\Q_p^N\to \Q_p^N.$$ This means that $\phi_p\alpha_p(x)=\alpha_p\phi(x)$ for every $x\in\Q^N$. So $\phi_p$ and $\phi$ are represented by the same matrix.

\smallskip
We prove now some technical results that allow us to bound the $p$-adic contributions to the algebraic entropy of $\phi$ from above using the trajectories and from below using the minor trajectories of $\phi_p$ in $\Q_p^N$. 

We start with the following lemma, which applies in the proofs of
Propositions \ref{appr-p} and \ref{appr}.

\begin{lemma}\label{cubi-palle}
Let $p\in\p$, $k\in\N_+$ with $k\geq 3$, and let $h$ be the maximal non-negative integer such that $p^h\leq k$.
\begin{enumerate}[\rm (1)]
\item If $p<\infty$, then $D_{p,N}(k)\subseteq \alpha_p E_{m}+D_{p,N}(1)$  for every $m\in\N_+$ such that $p^h|m$.
\item For every $m\geq k$, $D_{\infty,N}(1)\subseteq \alpha_\infty E_m+D_{\infty,N}(1/k)$. 
\end{enumerate}
\end{lemma}
\begin{proof}
By Remark \ref{rm1}, it suffices to prove the result in the case $N=1$.

\1 Let $m\in\N_+$ be such that $p^h|m$. If $x\in D_p(k)$ then $|x|_p\leq p^h$ (see Remark \ref{rm}). Write the $p$-adic expansion
$$x=x_{-h} p^{-h}+ \ldots +x_{-1}p^{-1}+x_{0} +\ldots,$$
with $0\leq x_j< p$ for every $j=-h,\ldots,0,\ldots$. Let now $y$ be a rational number with ``bounded" $p$-adic expansion of the form
$$y=x_{-h} p^{-h}+ \ldots +x_{-1} p^{-1}. $$
  It is then clear that $x-y\in D_p(1)$. Since $y$ is a rational number with denominator $p^h$ and numerator between $0$ and $hp^h$,  there exists $z\in \Z\subseteq D_p(1)$ such that $y-z=w$ has denominator $p^h$ and numerator between $0$ and $p^h$. Therefore, $w\in \alpha_p E_m$ and $x-w\in D_p(1)$, that gives $x\in\alpha_p E_m+D_p(1)$ as desired.

\smallskip
\2 This is clear.
\end{proof}

\begin{proposition}\label{appr-p}
Let $m,n\in\N_+$ and let $p\in\p$ be finite. Then:
\begin{enumerate}[\rm (1)]
\item $T^{p}_n(\phi,E_m)\subseteq T_n(\phi_p,D_{p,N}(|1/m|_p))$;
\item $T_n^\leq (\phi_p,D_{p,N}(1)) \subseteq T^{p}_n(\phi,E_{m})$, if $||\phi_p||_p$ divides $m$.
\end{enumerate}
\end{proposition}
\begin{proof}
\1 Since 
$ \alpha_p T_n(\phi,E_m)=\alpha_p(E_m+\ldots+\phi^{n-1}E_m)=\alpha_p E_m +\ldots+\phi_p^{n-1}\alpha_p E_m$, since
$\alpha_p E_m\subseteq D_{p,N}(|1/m|_p)$, and as $D_{p,N}(1)+D_{p,N}(|1/m|_p)=D_{p,N}(|1/m|_p)$ by the strong triangular inequality, we obtain
$$T^{p}_n(\phi,E_m)= \alpha_p T_n(\phi,E_m)+D_{p,N}(1)\subseteq T_n(\phi_p,D_{p,N}(|1/m|_p))+D_{p,N}(1)\subseteq T_n(\phi_p,D_{p,N}(|1/m|_p)) .$$

\noindent
\2 
We use induction on $n\geq 1$ to prove that $T_n^\leq (\phi_p,D_{p,N}(1))  \subseteq T^{p}_n(\phi,E_{m})$. 

For $n=1$ it is enough to notice that $T_1^\leq (\phi_p,D_{p,N}(1))=D_{p,N}(1)$ is clearly contained in $T_1^{p}(\phi,E_m)$. So let us prove that, if
\begin{equation}\label{x} 
T_n^\leq (\phi_p,D_{p,N}(1)) \subseteq T^{p}_n(\phi,E_m)
\end{equation}
for some $n\in\N_+$ then $T_{n+1}^\leq (\phi_p,D_{p,N}(1)) \subseteq T^{p}_{n+1}(\phi,E_m)$. In particular, we need to show that
$$x+D_{p,N}(1)\subseteq T^{p}_{n+1}(\phi,E_m)\ \text{for every}\ x\in \phi_p^{n}D_{p,N}(1). $$
Indeed, given $x\in \phi_p^{n}D_{p,N}(1)$, there exists $y\in \phi_p^{n-1}D_{p,N}(1)$ such that $\phi_p(y)=x$. 
Furthermore, by \eqref{x} we have that $y\in T^{p}_n(\phi,E_m)$, and so there exists $z\in T_n(\phi,E_m)$ such that $y\in \alpha_p(z)+D_{p,N}(1)$. This shows that 
\begin{equation}\label{x''}
x=\phi_p(y)\in \phi_p(\alpha_p(z))+\phi_p D_{p,N}(1)\subseteq \alpha_p(\phi(z))+D_{p,N}(||\phi_p||_p) ,
\end{equation}
by \eqref{imm-palle}. As we supposed that $||\phi_p||_p$ divides $m$, we can use Lemma \ref{cubi-palle}(1) to show that
\begin{equation}\label{x'''} 
D_{p,N}(||\phi_p||_p)\subseteq \alpha_p E_m+D_{p,N}(1).
\end{equation}
Now \eqref{x''} and \eqref{x'''} together give
$$x\in \alpha_p(\phi(z))+D_{p,N}(||\phi_p||_p)\subseteq \alpha_p \phi T_n(\phi,E_m)+ \alpha_p E_m+D_{p,N}(1) =\alpha_p T_{n+1}(\phi,E_m)+D_{p,N}(1).$$
So, $x+ D_{p,N}(1)\subseteq \alpha_p T_{n+1}(\phi,E_m)+D_{p,N}(1)+ D_{p,N}(1)=T^{p}_{n+1}(\phi,E_m)$.
\end{proof}  

The following proposition is the counterpart of Proposition \ref{appr-p} for $p=\infty$.

\begin{proposition}\label{appr}
Let $k=\max\{\lceil||\phi_\infty||_\infty+1\rceil,\ 3\}$ and $m,n\in\N_+$ with $m\geq k$. Then:
\begin{enumerate}[\rm (1)]
\item $T^{\infty}_n(\phi,E_m,k)\subseteq T_n(\phi_\infty,D_{\infty,N}(2))$;
\item $T_n^\leq(\phi_\infty,D_{\infty,N}(1/k))\subseteq T^\infty_n(\phi,E_{m},k)$.
\end{enumerate}
\end{proposition}
\begin{proof}
(1) Since 
$$\alpha_\infty T_n(\phi,E_m) =\alpha_\infty(E_m+\ldots+\phi^{n-1}E_m)=\alpha_\infty E_m +\ldots+\phi_\infty^{n-1}\alpha_\infty E_m,$$ 
$\alpha_\infty E_m\subseteq D_{\infty,N}(1)$ and $D_{\infty,N}(1)+D_{\infty,N}(1/k)\subseteq D_{\infty,N}(2)$, we easily obtain that
$$T^{\infty}_n(\phi,E_m,k)= \alpha_\infty T_n(\phi,E_m) +D_{\infty,N}(1/k)\subseteq T_n(\phi_\infty,D_{\infty,N}(1))+D_{\infty,N}(1/k)\subseteq T_n(\phi_\infty,D_{\infty,N}(2)).$$

\smallskip
(2) 
Firstly we verify that 
\begin{equation}\label{y'}
\phi_\infty D_{\infty,N}(1/k)+D_{\infty,N}(1/k)\subseteq D_{\infty,N}(1).
\end{equation}
Given $x\in \phi_\infty D_{\infty,N}(1/k)+D_{\infty,N}(1/k)$,  there exist $x_1,x_2\in D_{\infty,N}(1/k)$ such that $x=\phi_\infty (x_1)+x_2$. Thus we obtain
$$|x|_\infty\leq |\phi_\infty(x_1)|_\infty+  |x_2|_\infty\leq ||\phi_\infty||_\infty|x_1|_\infty + 1/k \leq (1+||\phi_\infty||_\infty)1/k.$$
Since $(1+||\phi_\infty||_\infty)1/k\leq 1$, as by hypothesis $k\geq ||\phi_\infty||_\infty+1$, it follows that \eqref{y'} holds true.

We now use induction on $n\geq 1$ to prove the thesis of (2), that is $T_n^\leq(\phi_\infty,D_{\infty,N}(1/k)) \subseteq T^{\infty}_n(\phi,E_{m},k)$. For $n=1$ it is enough to notice that $D_{\infty,N}(1/k)+D_{\infty,N}(1/k)\subseteq D_{\infty,N}(1)\subseteq T_1^{\infty}(\phi,E_m,k)$ by Lemma \ref{cubi-palle}(2). 

So let us prove that, if
\begin{equation}\label{y} 
T_n^\leq(\phi_\infty,D_{\infty,N}(1/k)) \subseteq T^{\infty}_n(\phi,E_m,k)
\end{equation}
for some $n\in\N_+$ then $T_{n+1}^\leq(\phi_\infty,D_{\infty,N}(1/k)) \subseteq T^{\infty}_{n+1}(\phi,E_m,k)$. 
In particular, we need to show that
$$x+D_{\infty,N}(1/k)\subseteq T^{\infty}_{n+1}(\phi,E_m,k)\ \text{for every}\ x\in \phi_\infty^{n}D_{\infty,N}(1/k).$$
 To this end, let $x\in \phi_\infty^{n} D_{\infty,N}(1/k)$; then there exists $y\in \phi_\infty^{n-1}D_{\infty,N}(1/k)$ such that $\phi_\infty(y)=x$. 
By \eqref{y} we have that $y\in T^{\infty}_n(\phi,E_m,k)$ and so there exists $z\in T_n(\phi,E_m)$ such that $y\in\alpha_\infty(z)+D_{\infty,N}(1/k)$. This shows that 
$$x=\phi_\infty(y)\subseteq \phi_\infty(\alpha_\infty(z))+\phi_\infty D_{\infty,N}(1/k).$$
Hence,
\begin{equation}\label{y''}
x+D_{\infty,N}(1/k)\subseteq \phi_\infty(\alpha_\infty(z))+\phi_\infty D_{\infty,N}(1/k)+D_{\infty,N}(1/k)\subseteq\alpha_\infty(\phi(z))+D_{\infty,N}(1),
\end{equation}
by \eqref{y'}. By Lemma \ref{cubi-palle}(2),
\begin{equation}\label{y'''} 
D_{\infty,N}(1)\subseteq \alpha_\infty E_m+D_{\infty,N}(1/k).
\end{equation}
Now \eqref{y''} and \eqref{y'''} yield
\begin{equation*}\begin{split}
x+D_{\infty,N}(1/k)\subseteq \alpha_\infty(\phi(z))+D_{\infty,N}(1)\subseteq \alpha_\infty\phi T_n(\phi,E_m)+ \alpha_\infty E_m+D_{\infty,N}(1/k) \\=\alpha_\infty T_{n+1}(\phi,E_m)+D_{\infty,N}(1/k)=T^{\infty}_{n+1}(\phi,E_m,k),
\end{split}\end{equation*}
as desired.
\end{proof}

Consider the following infinite set of natural numbers:
$$\mathcal N_1(\phi)=\left\{m\in\N_+ :  m=c\cdot \prod_{p\in\Pi(\phi,1)}||\phi_p||_p,\ \text{ with }c\geq\max\{\lceil||\phi_\infty||_\infty+1\rceil,\ 3\}\right\};$$
with the convention that an empty product is equal to $1$.
Then every $m\in \mathcal N_1(\phi)$ satisfies all the hypotheses of Propositions \ref{appr-p} and \ref{appr},
and $\{E_m:m\in\mathcal N_1(\phi)\}$ is cofinal in $\{E_m:m\in\N_+\}$ since $\mathcal N_1(\phi)$ is cofinal in $\N_+$.

\smallskip
As announced, we can now prove that the $p$-adic contribution $H^p(\phi,E_m)$ to the algebraic entropy of an endomorphism $\phi$ of $\Q^N$ is the algebraic entropy of $\phi_p$. 

\begin{proposition}\label{contributi=entropia}
Let $m\in\mathcal N_1(\phi)$ and $p\in\p$. Then
$$h_A(\phi_p)=H^p(\phi,E_m).$$
\end{proposition} 
\begin{proof}
We give a proof in case $p<\infty$. The case of $p=\infty$ is completely analogous. Since our choice of $m$ satisfies the hypotheses of Proposition \ref{appr-p} we get
\begin{equation}\label{fff}
T_n^\leq (\phi_p,D_{p,N}(1)) \subseteq T^{p}_n(\phi,E_{m})\subseteq T_n(\phi_p,D_{p,N}(|1/m|_p)) .\end{equation}
Choose a Haar measure $\mu$ on $\Q_p^N$. Applying $\mu$, taking logarithms and passing to the $\limsup$ in \eqref{fff} we get
$$h_A(\phi_p)=H^\leq (\phi_p,D_{p,N}(1)) \leq H^{p}(\phi,E_{m})\leq H_A(\phi_p,D_{p,N}(|1/m|_p))=h_A(\phi_p) ,$$
where the first and the last equality follow from Theorem \ref{yuz-pp}.
\end{proof}

\subsection{Algebraic entropy as sum of  $p$-adic contributions}
In this section, and more precisely in Theorem \ref{contributi} below, we come to a proof of the main result applied in the Algebraic Yuzvinski Formula. We start with two technical lemmas which permit some control respectively on the euclidean and the $p$-adic part of the diagonal embedding $\alpha_{\P}:\Q^N\to \prod_{p\in\P}\Q_p^N$ for some finite set of primes $\P$ containing $\infty$.

\begin{lemma}\label{parte_infinita}
Let $h$ and $k$ be positive integers and $\P=\{p_1,\dots,p_h\}$ be a set of primes. There exists ${\bar m}={\bar m}(k,\P)$ such that any multiple $m$ of ${\bar m}$ has the following property: 
\begin{itemize}
\item[($*$)] given $y_1\in E_m$ and $y_2\in [-1,1]$, there exists $y\in E_m$ such that
\begin{enumerate}[\rm (1)]
\item $y-y_1=c/d$ with $(p_i,d)=1$ for all $i=1,\dots,h$;
\item $|y-y_2|<1/k$.
\end{enumerate}
\end{itemize}
\end{lemma}
\begin{proof}
Let $p\notin \P$ be a prime such that $p\geq 2k+1$, define $\bar m=p$,
choose arbitrarily a multiple $m$ of $\bar m$ and let us prove that $m$ verifies ($*$). Indeed, let $y_1\in E_m$ and $y_2\in [-1,1]$. Let $j_1,j_2\in [-p+1,p]$ be two integers such that $y_1\in [(j_1-1)/p,j_1/p]$ and $y_2\in [(j_2-1)/p,j_2/p]$. Let $c/d=(j_2-j_1)/p$ and $y=y_1+c/d$. We  have to verify that such $y$ belongs to $E_m$ and satisfies (1) and (2). Now, (1) follows by the choice of $p\notin \P$. While (2) follows by the following computation:
$$|y-y_2|= |y_1-j_1/p+j_2/p-y_2|\leq |y_1-j_1/p|+|j_2/p-y_2|\leq 2/p< 1/k\, .$$
In order to prove that $y\in E_m$ we have to show that $y=a/b$, where $b$ divides $m$ and $|a|\leq |b|$. But in fact, $b$ can be chosen to be the minimum common multiple of $p$ and the denominator of $y_1$, both dividing $m$. Furthermore, $$y= y_1+c/d\leq j_1/p+ (j_2-j_1)/p=j_2/p\leq 1$$ and $$y= y_1+c/d\geq (j_1-1)/p+ (j_2-j_1)/p=(j_2-1)/p\geq -1,$$ that is, $|y|\leq 1$ as desired.
\end{proof}

\begin{lemma}\label{parte_finita} 
Let $k\geq 3$ and $h$ be positive integers and $\P=\{p_1,\dots,p_h\}$ be a set of primes. There exists $\bar{\bar m}=\bar{\bar m}(k,\P)$ such that any multiple $m$ of $\bar{\bar m}$ has the following property:
\begin{itemize}
\item[($**$)] given $x_i\in D_{p_i}(k)$ with $i=1,\dots,h$, there exists $y\in E_m$ such that $x_i\in \alpha_{p_i}(y)+D_{p_i}(1)$ for all $i=1,\dots, h$.
\end{itemize}
\end{lemma}
\begin{proof} 
We proceed by induction on $h$. 

\smallskip
If $h=1$, then $\P=\{p_1\}$. Let $l$ be the maximal non-negative integer such that $p_1^{l}\leq k$ and define $\bar{\bar m}_1=p_1^l$. Given any multiple $m$ of $\bar{\bar m}_1$, we obtain that $x_1\in \alpha_{p_1}(E_{\bar{\bar m}_1})+D_{p_1}(1)\subseteq \alpha_{p_1}(E_{m})+D_{p_1}(1)$, by Lemma \ref{cubi-palle}. 

\smallskip
Let now $h\geq 1$ and suppose that there exists $\bar{\bar m}_h$ whose multiples satisfy ($**$); moreover, let  $l$ be the maximal non-negative integer such that $p_{h+1}^{l}\leq k$, define $\bar{\bar m}_1=p_{h+1}^l$ and let $$m_h=\bar{\bar m}_h\bar{\bar m}_1.$$

Let $t=p_1\cdot\ldots\cdot p_h$ and choose a positive integer $j$ such that $2m_h^2\leq t^j$. We let $\bar{\bar m}_{h+1}=(t^jm_h^2)!$ and we take $m$ to be a multiple of $\bar{\bar m}_{h+1}$. 
Given $x_i\in D_{p_i}(k)$ with $i=1,\dots,h+1$, we have to show that there exists $y\in E_m$ such that 
\begin{equation}\label{tesi_finita_h+1}x_i\in \alpha_{p_i}(y)+D_{p_i}(1)\,,\ \ \ \text{ for all }i=1,\dots, h+1\,.\end{equation} 
By inductive hypothesis, there exists $y'\in E_{m_h}$  such that $x_i\in \alpha_{p_i}(y')+D_{p_i}(1)$ for all $i=1,\dots, h$ and, by Lemma \ref{cubi-palle}, there is $y''\in E_{\bar{\bar m}_1}$ such that $x_{h+1}\in \alpha_{p_{h+1}}(y'')+D_{p_{h+1}}(1)$. Let us show that there exist two coprime integers $c$ and $d$ such that
\begin{itemize}
\item[(a)] $(p_{h+1},d)=1$; 
\item[(b)] letting $y=y''+c/d$, we have that $y\in E_m$;
\item[(c)] letting $a/b=y-y'$, then $(b,p_i)=1$ for all $i=1,\dots,h$.
\end{itemize}
Notice that, provided such $c$ and $d$ exist, one has that 
$$x_i\in \alpha_{p_i}(y')+D_{p_i}(1)= \alpha_{p_i}(y')+\alpha_{p_i}(a/b)+D_{p_i}(1)=\alpha_{p_i}(y)+D_{p_i}(1)$$ 
as, by part (c), $\alpha_{p_i}(a/b)\in D_{p_i}(1)=\J_{p_i}$, since this group is $q$-divisible for all $q\neq p_i$, for all $i=1,\dots,h$. 
Furthermore, $$x_{h+1}\in \alpha_{p_{h+1}}(y'')+D_{p_{h+1}}(1)= \alpha_{p_{h+1}}(y'')+\alpha_{p_{h+1}}(c/d)+D_{p_{h+1}}(1)=\alpha_{p_{h+1}}(y)+D_{p_{h+1}}(1)$$ as, by part (a), $\alpha_{p_{h+1}}(c/d)\in D_{p_{h+1}}(1)$. Thus, the existence of two integers $c$ and $d$ satisfying (a), (b) and (c) implies \eqref{tesi_finita_h+1}, concluding the proof.

Hence, let us find such $c$ and $d$. Indeed, let $y''-y'=a'/b'$. Decompose $b'$ as a product of primes, and write $b'=b_1b_2b_3$, where $(b_1,p_i)=1$ for all $i=1,\dots,h+1$, $b_2$ is a power of $p_{h+1}$ and $b_3$ is a product of powers of the $p_i$, with $i=1,\dots,h$. We distinguish three cases:
\begin{enumerate}[{\rm Case} $1$.]
\item If $b_2=1$, then $(p_h,b')=1$ and so we can conclude just letting $c=-a'$ and $d=b'$ (so that $y=y'\in E_{m_h}\subseteq E_m$, $a=0$ and $b=1$).
\item If either $y''=1$ or $y''=-1$, then notice that $x_{h+1}\in \alpha_{p_{h+1}}(y'')+D_{p_{h+1}}(1)=D_{p_{h+1}}(1)=\alpha_{p_{h+1}}(0)+D_{p_{h+1}}(1)$. Thus we can exclude this case just changing our choice for $y''$, that is, taking $y''=0\in E_{\bar {\bar m}_1}$.
\item  If $b_2\neq 1$ and $y''\neq\pm1$, then let $d=b_3(t^jb_3-b_1b_2)$ and $c=a'$. It follows that
\begin{align*}\frac{a}{b}&=\frac{a'}{b'}+\frac{c}{d}=\frac{a'd+b'c}{b'd}=\frac{a'b_3(t^jb_3-b_1b_2)+b_1b_2b_3a'}{b_1b_2b_3^2(t^jb_3-b_1b_2)}\\ 
&=\frac{a'b_3(t^jb_3-b_1b_2+b_1b_2)}{b_1b_2b_3^2(t^jb_3-b_1b_2)}=\frac{a't^j}{b_1b_2(t^jb_3-b_1b_2)}\, .\end{align*}
Let us verify (a), (b) and (c). Indeed, $(p_{h+1},d)=(p_{h+1},b_3)(p_{h+1},t^jb_3-b_1b_2)$ and $(p_{h+1},b_3)=1$ by definition of $b_3$. Furthermore, as $b_2\neq 1$, then $p_{h+1}$ divides $b_1b_2$ but, by definition of $t$ and $b_3$, $(p_{h+1},t^jb_3)=1$, thus $(p_{h+1}, t^jb_3-b_1b_2)=1$, as required by (a). To prove (b), we have to show that $y''+c/d\in E_m$. Notice that, given two elements $e_1,e_2\in E_m$, then $e_1+e_2\in E_m$ if and only if $|e_1+e_2|\leq 1$. Clearly $y''\in E_m$ while $c/d\in E_m$ since $d=b_3(t^jb_3-b_1b_2)=t^jb_3^2-b'\leq t^j(b')^2\leq t^jm_h^2$, so $d$ divides $m$, and 
$$|c/d|=\frac{|a'|}{|b_3(t^jb_3-b_1b_2)|}\leq \frac{2m_h}{t^j}\leq \frac{2m_h}{2m_h^2}= \frac{1}{m_h}\,.$$
Now, the fact that $|y''+c/d|\leq 1$ follows by the fact that we assumed that $y''\neq\pm1$ and so $|y''|\leq(m_h-1)/m_h$, thus $|y''+c/d|\leq (m_h-1)/m_h+1/m_h=1$.
Finally, to prove (c), it is enough to verify that $b_1b_2(t^jb_3-b_1b_2)$ is coprime with $p_i$, for all $i=1,\dots, h$. Choose arbitrarily $i\in\{1,\dots,h\}$, then $(p_i, b_1b_2)=1$ by construction. Furthermore, $(p_i,t^jb_3-b_1b_2)=1$ as $p_i$ divides $t^jb_3$ but not $b_1b_2$.
\end{enumerate}
\end{proof}

Given an endomorphism $\phi:\Q^N\to \Q^N$, let 
$$\Phi=\prod_{p\in\P(\phi)}\phi_p:\prod_{p\in\P(\phi)}\Q^N_p\to \prod_{p\in\P(\phi)}\Q^N_p\, ,$$
where $\phi_p=\phi\otimes_\Q id_{\Q_p^N}$. 
Let 
$$k=\max(\{||\phi_{p}||_{p}:p\in\P(\phi)\}\cup\{\lceil||\phi_\infty||_\infty+1\rceil,\ 3\})\,,$$
we define the following set of positive integers 
$$\mathcal N(\phi)=\{c\bar m\bar{\bar m}:c\in \mathcal N_1(\phi)\}\, ,$$
where $\bar m=\bar m(k,\Pi(\phi))$ and $\bar{\bar m}=\bar{\bar m}(k,\Pi(\phi))$ are the positive integers given by Lemmas \ref{parte_infinita} and \ref{parte_finita} respectively. Notice that $\mathcal N(\phi)\subseteq \mathcal N_1(\phi)$ and it is cofinal in $\N_+$. 

We apply now the previous two lemmas to prove the following proposition, which is fundamental for proving Theorem \ref{contributi} below. In particular, item (3) shows that, taken $m\in\mathcal N(\phi)$, inside the diagonal embedding $\alpha_\P(T_n(\phi,E_{m}))$ of the $n$-th $\phi$-trajectory of $E_m$, enlarged adding the disk $\D_\P(1/k)$, one can find the $n$-th minor $\Phi$-trajectory $T_n^\leq(\Phi,\D_\P(1/k))$ of the same disk $\D_\P(1/k)$.

\begin{proposition}\label{appr***}
Let $\P=\P(\phi)$ and $\Pi=\Pi(\phi)$, let also $m,n\in\N_+$ with $m\in\mathcal N(\phi)$. Then:
\begin{enumerate}[\rm (1)]
\item $D_{\infty,N}(1)\times\prod_{p\in\P^{<\infty}}D_{p,N}(k)\subseteq \alpha_\P(E_m)+\D_{\P}(1/k)$;
\item $\D_\P(1/k)+\Phi\D_\P(1/k)\subseteq \prod_{p\in\P^{<\infty}}D_{p,N}(k)\times D_{\infty,N}(1)$;
\item $T_n^\leq(\Phi,\D_\P(1/k))\subseteq \alpha_\P T_n(\phi,E_{m})+\D_\P(1/k)$.
\end{enumerate}
\end{proposition}
\begin{proof}
(1) All the sets involved are rectangular so it is enough to exhibit a proof in case $N=1$. 

Let $h=|\Pi|$ and 
$\Pi=\{p_1,\dots,p_h\}$. Consider an arbitrary element $x\in \prod_{p\in\Pi}D_{p}(k)\times D_{\infty}(1)$. Denote by $x_i\in D_{p_i}(k)$ the $p_i$-th component of $x$, $i=1,\dots,h$. By Lemma \ref{parte_finita}, there exists $y'\in E_m$ such that $x_i\in \alpha_{p_i}(y)+D_{p_i}(1)$ for all $i=1,\dots, h$. Now, denote by $x_\infty\in D_{\infty}(1)=[-1,1]$ the $\infty$-th component of $x$. By Lemma \ref{parte_infinita}, there exists $y\in E_m$ such that $y-y'=c/d$ with $(p_i,d)=1$ for all $i=1,\dots,h$ and $|y-x_\infty|\leq1/k$. This means that 
$$x_i\in \alpha_{p_i}(y')+D_{p_i}(1)=\alpha_{p_i}(y')+\alpha_{p_i}(c/d)+D_{p_i}(1)=\alpha_{p_i}(y)+D_{p_i}(1)$$ for all $i=1,\dots, h$ and $x_\infty\in \alpha_{\infty}(y)+D_\infty(1/k)$. This shows that $x\in \alpha_\P(y)+\D_\P(1/k)$, as desired.

\smallskip
(2) Given $x\in \D_{\P}(1/k)+\Phi\D_\P(1/k)$ let $x_i\in D_{p_i,N}(1)+\phi_{p_i}D_{p_i,N}(1)$ and $x_\infty\in D_{\infty,N}(1/k)+\phi_{\infty}D_{\infty,N}(1/k)$ be the components of $x$. The there exist $y_1^{(i)},y_2^{(i)}\in D_{p_i,N}(1)$ for all $i=1,\dots, h$ and $y_1^{(\infty)},y_2^{(\infty)}\in D_{p_i,N}(1/k)$ such that $x_i=\phi_{p_i} (y^{(i)}_1)+y^{(i)}_2$ for all $i=1,\dots, h$ and $x_\infty=\phi_{\infty} (y^{(\infty)}_1)+y^{(\infty)}_2$. Thus we obtain
$$|x_i|_{p_i}\leq \max\left\{|\phi_{p_i} (y^{(i)}_1)|_{p_i}+  |y^{(i)}_2|_{p_i}\right\}\leq\max\left\{ ||\phi_{p_i}||_{p_i}|y^{(i)}_1|_{p_i}, 1 \right\}\leq k\, .$$
$$|x_\infty|_{\infty}\leq |\phi_{\infty} (y^{(\infty)}_1)|_{\infty}+  |y^{(\infty)}_2|_{\infty}\leq ||\phi_{\infty}||_{\infty}|y^{(\infty)}_1|_{\infty} + 1/k \leq (1/k+||\phi_{\infty}||_{\infty}/k)\leq 1\, .$$

\smallskip
(3) We use induction on $n\geq 1$. For $n=1$ it is enough to notice that 
$$\D_{\P}(1/k)+\D_{\P}(1/k)\subseteq D_{\infty,N}(1)\times \prod_{p\in\P^{<\infty}}D_{p,N}(k)\subseteq \alpha_\P E_m+\D_\P(1/k)$$ by part (1). 

So let us prove that, if
\begin{equation}\label{y*} 
T_n^\leq(\Phi,\D_{\P}(1/k)) \subseteq \alpha_\P T_n(\phi,E_m)+\D_\P(1/k)
\end{equation}
for some $n\in\N_+$, then $T_{n+1}^\leq(\Phi,\D_{\P}(1/k)) \subseteq \alpha_\P T_{n+1}(\phi,E_m)+\D_\P(1/k)$. 
In particular, we need to show that
$$x+\D_{\P}(1/k)\subseteq \alpha_\P T_{n+1}(\phi,E_m)+\D_\P(1/k)\ \text{for every}\ x\in \Phi^{n}\D_{\P}(1/k).$$
 To this end, let $x\in \Phi^{n} \D_{\P}(1/k)$; then there exists $y\in \Phi^{n-1}\D_{\P}(1/k)$ such that $\Phi(y)=x$. 
By \eqref{y*} we have that $y\in \alpha_\P T_n(\phi,E_m)+\D_\P(1/k)$ and so there exists $z\in T_n(\phi,E_m)$ such that $y\in\alpha_\P(z)+\D_{\P}(1/k)$. This shows that 
$$x=\Phi(y)\in \Phi(\alpha_\P(z))+\Phi \D_{\P}(1/k).$$
Hence,
\begin{equation}\label{y''*}
x+\D_{\P}(1/k)\subseteq \Phi(\alpha_\P(z))+\Phi \D_{\P}(1/k)+\D_{\P}(1/k)\subseteq\alpha_\P(\phi(z))+D_{\infty,N}(1)\times \prod_{p\in\P^{<\infty}}D_{p,N}(k),
\end{equation}
by part (2). Applying again part (1) we obtain
\begin{equation*}\begin{split}
x+\D_{\P}(1/k)\subseteq \alpha_\P(\phi(z))+D_{\infty,N}(1)\times\prod_{p\in\P^{<\infty}}D_{p,N}(k)\subseteq \alpha_\P\phi T_n(\phi,E_m)+ \alpha_\P E_m+\D_{\P}(1/k) \\=\alpha_\P T_{n+1}(\phi,E_m)+\D_{\P}(1/k),
\end{split}\end{equation*}
as desired.
\end{proof}

Finally, applying Propositions \ref{traj-sep} and \ref{coro-mis} from the previous subsections, and Proposition \ref{appr***} above we can prove the following theorem, which implies in particular Fact B of the Introduction; it expresses the algebraic entropy of $\phi:\Q^N\to \Q^N$ as the sum of the $p$-adic contribution to the algebraic entropy of $\phi$ at $E_m$ for every $m\in\mathcal N(\phi)$ and so, thanks to Proposition \ref{contributi=entropia}, also as the sum of the algebraic entropy of the endomorphisms $\phi_p=\phi\otimes_{\Q}id_{\Q_p}:\Q_p^N\to \Q_p^N$.

\begin{theorem}\label{contributi}
Given an endomorphism $\phi:\Q^N\to \Q^N$, the following equalities hold true for all $m\in \mathcal N(\phi)$:
\begin{equation}\label{contributi-p}
h_A(\phi)=H_A(\phi,E_m)=\sum_{p\in\p}H^p(\phi,E_m)=\sum_{p\in \p}h_A(\phi_p) .
\end{equation}
\end{theorem}
\begin{proof}
The last equality is given by Proposition \ref{contributi=entropia}; we prove the other two.

For every finite prime $p$, let  $\mu_p$ be the Haar measure on $\Q_p^N$ such that $\mu_p(D_{p,N}(1))=1$, $\mu_{\infty}$ the Haar measure on $\Q_\infty^N$ such that $\mu_\infty(D_{\infty,N}(1/k))=1$, and let $\mu$ be the Haar measure on $\prod_{p\in\P(\phi,m)}\Q_p^N$ such that $\mu(\D_{\P(\phi,m)}(1/k))=1$. 

We assume $\Pi(\phi,m)$ to be non-empty; in case $\P(\phi,m)=\{\infty\}$ a similar argument leads to the same conclusion.
So fix $n\in\N_+$. By Proposition \ref{traj-sep}(2), $T_n(\phi,E_m)\subseteq\Z^N_{(\P(\phi,m))}$. Hence, Proposition \ref{coro-mis}(2) gives
\begin{equation}\label{ee}
|T_n(\phi,E_m)|=\mu(T^*_n(\phi,E_m,k)).
\end{equation}
Using \eqref{ee} and the fact that $T_n^*(\phi,E_m,k)\subseteq   T_n^{\infty}(\phi,E_m,k)\times\prod_{p\in\Pi(\phi,m)}T_n^{p}(\phi,E_m), $
we obtain
$$|T_n(\phi,E_m)|\leq \mu_\infty( T_n^{\infty}(\phi,E_m,k))\cdot \prod_{p\in\Pi(\phi,m)}\mu_{p}\left(T_n^{p}(\phi,E_m)\right) .$$
Taking logarithms, dividing by $n$ and letting $n$ go to infinity we get
\begin{equation}\label{prima_dis_finale}
H_A(\phi,E_m)\leq \sum_{p\in\p}H^p(\phi,E_m),
\end{equation}
applying Lemma \ref{contributii}.

On the other hand, by Proposition \ref{appr***}(3), we have that 
$$T_n^\leq(\phi_\infty,D_{\infty,N}(1/k))\times \prod_{p\in \Pi(\phi)}T^{\leq}_n(\phi_p,D_{p,N}(1)) \subseteq \alpha_{\P(\phi)}(T_n(\phi,E_m))+\D_{\P(\phi)}(1/k)\, ,$$ for all $n\in\N_+$. Let $\mu_*$ be the product of the measures $\mu_p$ with $p\in\P(\phi)$. Taking logarithms of the measures, dividing by $n$ and letting $n$ go to infinity we get, applying Lemma \ref{contributii} for the first equality,
\begin{align}\label{seconda_dis_finale}
\notag\sum_{p\in\p}H^p(\phi,E_m)&=\sum_{p\in\P(\phi)}H^p(\phi,E_m)\\
&\notag=\limsup_{n\to\infty}\frac{1}{n}\log\mu_*\left(T_n^\leq(\phi_\infty,D_{\infty,N}(1/k))\times\prod_{p\in \Pi(\phi)}T^{\leq}_n(\phi_p,D_{p,N}(1))\right)\\
&\leq \limsup_{n\to\infty}\frac{1}{n}\log\mu_*(\alpha_{\P(\phi)}(T_n(\phi,E_m))+\D_{\P(\phi)}(1/k))\\
&\notag\leq \limsup_{n\to\infty}\frac{1}{n}\log(|T_n(\phi,E_m)|\mu_*(\D_{\P(\phi)}(1/k)))=H_A(\phi,E_m).
\end{align}
By \eqref{prima_dis_finale} and \eqref{seconda_dis_finale} we can conclude that, for every $m\in\mathcal N(\phi)$,
$$H_A(\phi,E_m)=\sum_{p\in\p}H^p(\phi,E_m).$$
Finally, using the fact that $\mathcal N(\phi)$ is cofinal in $\N_+$, one obtains that the family $\{E_m:m\in\mathcal N(\phi)\}$ is cofinal in $\{E_m:m\in\N_+\}$ and so, for every $m\in\mathcal N(\phi)$,
\begin{align*}h_A(\phi)&=\sup \{H_A(\phi,E_m):m\in\mathcal N(\phi)\}\\ 
&=\sup \left\{\sum_{p\in\p}H^p(\phi,E_m):m\in\mathcal N(\phi)\right\}\\
&=\sum_{p\in \p}h_A(\phi_p) =\sum_{p\in\p}H^p(\phi,E_m)=H_A(\phi,E_m),\end{align*}
and this concludes the proof.
\end{proof}

\subsection{Comparison with the topological case}\label{LW-sub}

In the proof of the classical Yuzvinski Formula given in \cite{LW} a fundamental step is \cite[Theorem 1]{LW} (see \eqref{LWB} below), in the same way as Theorem \ref{contributi} is for the proof of the Algebraic Yuzvinski Formula. 
In this subsection we compare their proofs to stress the difference. 

\medskip
Following \cite{LW} and \cite{We}, consider the \emph{adele ring} $\Q_\mathbb A$ of $\Q$, that is the restricted product $$\Q_\mathbb A=\left\{x\in\prod_{p\in\p}\Q_p:|x_p|_p\leq 1\ \text{for all but a finite number of }p\right\}.$$
For a finite $\P\subseteq \p$, let $$\Q_\mathbb A(\P)=\left\{x\in\Q_\mathbb A:|x_p|_p\leq 1 \ \text{if}\ p\not\in\P\right\}.$$
Then $\Q_\mathbb A=\bigcup_{\P\subseteq\p\ \text{finite}}\Q_\mathbb A(\P)$.

Each $\Q_\mathbb A(\P)$ endowed with the product topology inherited from $\prod_{p\in\p}\Q_p$ is locally compact. 
In particular, if $\P$ contains $\infty$, then $\Q_\mathbb A(\P)$ is an LCA group isomorphic to $\prod_{p\in\P}\Q_p\times \prod_{p\in\p\setminus\P}\Z_p$.
Furthermore, $\Q_\mathbb A$ with the coarsest topology making each of the $\Q_\mathbb A(\P)$ open is locally compact as well.
 
Let $\alpha:\Q\to\Q_\mathbb A$ be the diagonal embedding of $\Q$ in $\Q_\mathbb A$. Then $\alpha(\Q)$ is discrete in $\Q_\mathbb A$ and $\Q_\mathbb A/\alpha(\Q)$ is topologically isomorphic to $\widehat{\Q}$ (see \cite[Lemma 4.1]{LW}). Consequently, $\alpha^N(\Q^N)$ is discrete in $\Q_\mathbb A^N$ and $\Q_\mathbb A^N/\alpha^N(\Q^N)$ is topologically isomorphic to $\widehat{\Q}^N$.

Consider an automorphism $\psi:\widehat{\Q}^N\to\widehat{\Q}^N$. As noted in the Introduction, the action of $\psi$ on $\widehat\Q^N$ is given by an $N\times N$ matrix with coefficients in $\Q$. So the same matrix  induces an automorphism $\widetilde\psi:\Q_\mathbb A^N\to\Q_\mathbb A^N$ and automorphisms $\psi_p:\Q_p^N\to\Q_p^N$, for $p\in\p$, just extending the scalars.


In \cite{LW} Lind and Ward use the strong \cite[Theorem 20]{B} by Bowen (see Fact \ref{Fact-Bowen}) to see that $\psi$ and $\widetilde\psi$ have the same topological entropy.
As a second step they verify that the topological entropy of $\widetilde\psi$ coincides with the topological entropy of its restriction $\widetilde\psi_\mathcal P$ to $\Q_\mathbb A(\mathcal P)^N$, for a suitable finite subset $\P$ of $\p$ containing $\infty$ ($\P=\P(\phi,1)\cup\P(\phi^{-1},1)\cup\{\infty\}$). Then they see that the topological entropy of $\widetilde\psi_\mathcal P$ is equal to the topological entropy of the product $\prod_{p\in\P}\psi_p:\prod_{p\in\P}\Q_p^N\to \prod_{p\in\P}\Q_p^N$. 
On the other hand, $h_T(\psi_p)=0$ for every $p\in\p\setminus\P$.
In this way they come to a proof of their \cite[Theorem 1]{LW}, that is, 
\begin{equation}\label{LWB}
h_T(\psi)=\sum_{p\in\p}h_T(\psi_p).
\end{equation}

\medskip
It is not known whether the counterpart of \cite[Theorem 20]{B} holds for the algebraic entropy (see Problem \ref{pb-bowen} below). So in this section we have developed an ad-hoc technique to prove Theorem \ref{contributi}; we briefly recall it here.

First, given an endomorphism $\phi:\Q^N\to \Q^N$, we restrict to the subfamily $\{E_m:m\in\N_+\}$ of the finite subsets of $\Q^N$ to compute the algebraic entropy of $\phi$. Moreover, for $m\in\N_+$, we consider the natural subset $\P(\phi,m)$ of $\P$. Note that 
\begin{equation}\label{Z_P}
\Z_{(\P(\phi,m))}=\alpha^{-1}(\Q_\mathbb A(\P(\phi,m))).
\end{equation}
For every $p\in\p$ we define the $p$-adic contributions $H^p(\phi,E_m)$ to the algebraic entropy of $\phi$, using the embedding of $\Q$ in $\Q_p$.
It follows essentially from the definitions that $H^p(\phi,E_m)=0$ for every $p\in\p\setminus \P(\phi,1)$. Furthermore, letting $\phi_p=\phi\otimes_\Q id_{\Q_p}$ be the extension of $\phi$ to $\Q_p$, for all $p\in \p$, we are able to find a cofinal subset $\mathcal N_1(\phi)$ of the natural numbers such that 
\begin{equation}\label{non_so}h_A(\phi_p) =H^p(\phi,E_m)\ \ \ \ \text{ for every $p\in\p$ and $m\in\mathcal N_1(\phi)$.}\end{equation}
Embedding $\Q^N$ in the finite product $\prod_{p\in\P(\phi,m)}\Q_p^N$, which is an LCA group, and ``fattening'' the $n$-th $\phi$-trajectories $T_n(\phi,E_m)$ to measurable subsets $T_n^*(\phi,E_m)$ of $\prod_{p\in\P}\Q_p^N$, we manage to prove that $H_A(\phi,E_m)\leq \sum_{p\in\P(\phi,m)}H^p(\phi,E_m)=\sum_{p\in\p}H^p(\phi,E_m)$ for all $m\in\mathcal N_1(\phi)$.

On the other hand, we can find a cofinal subset $\mathcal N(\phi)$ of $\mathcal N_1(\phi)$ (which is therefore also cofinal  in $\N$), for which the minor trajectory $T_n^\leq(\Phi,\D_\P(1/k))$, where $\Phi=\prod_{p\in\P(\phi,1)}\phi_p:\prod_{p\in\P(\phi,1)}\Q_p\to\prod_{p\in\P(\phi,1)}\Q_p$, is contained in $\alpha_\P(T_n(\phi,E_{m}))+\D_\P(1/k)$ for all $n\in\N_+$ and $m\in\mathcal N(\phi)$. In this way we can show that $\sum_{p\in\p}H^p(\phi,E_m)=\sum_{p\in\P(\phi,m)}H^p(\phi,E_m)\leq H_A(\phi,E_m)$.

Finally, we can conclude the proof of Theorem \ref{contributi} just applying \eqref{non_so} to the above discussion:
\begin{align*}h_A(\phi)&=\sup \{H_A(\phi,E_m):m\in\mathcal N(\phi)\}=\sup \left\{\sum_{p\in\p}H^p(\phi,E_m):m\in\mathcal N(\phi)\right\}=\sum_{p\in \p}h_A(\phi_p). \end{align*}

\section{The $p$-adic contributions to Mahler measure}\label{conc}

Thanks to Theorem \ref{contributi} in the previous section we can see the algebraic entropy $h_A(\phi)$ of an endomorphism $\phi$ of $\Q^N$ as sum of the algebraic entropies $h_A(\phi_p)$ with $\phi_p=\phi\otimes_\Q id_{\Q_p}$ and $p$ ranging in $\p$, passing in this way from $\Q^N$ to $\Q_p^N$. Moreover, Theorem \ref{yuz-pp} computes the algebraic entropy $h_A(\phi_p)$ in terms of the eigenvalues of $\phi_p$ in a finite extension $K_p$ of $\Q_p$.

This subsection is devoted to a similar description of Mahler measure, which leads to the final proof of the Algebraic Yuzvinski Formula. We remark that the main ideas in this section come from \cite[Section 6]{LW}. Nevertheless we give here complete proofs in order to make the treatment clear and self-contained.

\smallskip
We start fixing some notations for the whole section. Let $$f(X)=X^N+ a_1 X^{N-1}+\ldots+ a_{N-1}X+a_N\in\Q[X]$$ be a monic polynomial. We denote by $s\in\N_+$ the minimum positive integer such that $$sf(X)\in\Z[X];$$ in particular, $sf(X)$ is primitive.

Fix $p\in\p$. Identifying $\Q[X]$ with a subring of $\Q_p[X]$, we can consider $f(X)$ as an element of $\Q_p[X]$. Let $\{\lambda^{(p)}_{i}:i=1,\ldots, N\}$ be the roots of $f(X)$ contained in some finite extension $K_p$ of $\Q_p$. This means that there is a factorization
$$f(X)=(X-\lambda_1^{(p)})\cdot\ldots\cdot(X-\lambda_N^{(p)}),$$
in other words, for every $n\in\{1,\ldots,N\}$ we have
\begin{equation}\label{coef}
a_n=\sum_{{i_1<\ldots< i_n}}\lambda^{(p)}_{i_1}\cdot\ldots\cdot\lambda^{(p)}_{i_n},
\end{equation}
where ${\{i_1,\ldots,i_n\}\subseteq\{1,\ldots,N\}}$. In what follows we always assume without loss of generality that the roots are ordered as
$$|\lambda^{(p)}_1|_p\geq |\lambda^{(p)}_2|_p\geq \ldots \geq |\lambda^{(p)}_N|_p.$$

Assume now that $p$ is finite. With this ordering, using \eqref{coef} and the non-Archimedean property of $|-|_p$, we get
\begin{equation}\label{coef'}
|a_n|_p=\left|\sum_{i_1<\ldots< i_n}\lambda_{i_1}^{(p)}\cdot\ldots\cdot\lambda_{i_n}^{(p)}\right|_p=|\lambda_1^{(p)}\cdot\ldots\cdot\lambda_n^{(p)}|_p.
\end{equation}

\smallskip
We go on proving two technical lemmas needed in the proof of Theorem \ref{m=mp}.

\begin{lemma}\label{>1}
Let $p\in\p$ be finite. Then $p$ divides $s$ if and only if $|\lambda^{(p)}_1|_p>1$.
\end{lemma}
\begin{proof}
If $|\lambda^{(p)}_1|_p\leq 1$, then \eqref{coef'} gives $|a_n|_p\leq 1$ for every $n=1,\ldots, N$. This means exactly that $p$ does not divide the denominator of $a_n$ for every $n=1,\ldots,N$. Hence $p$ cannot divide $s$.

On the other hand, suppose $|\lambda^{(p)}_1|_p>1$ and set $b_n=-s\cdot a_n$,  for every $n=1,\ldots, N$. Then
$$sf(X)=sX^N-b_1X^{N-1}-\ldots-b_N\in\Z[X] .$$ 
Since $\lambda_1^{(p)}$ is a root of $f(X)$, we get $s=b_1/\lambda^{(p)}_1+\ldots+b_N/(\lambda^{(p)}_1)^N$. Then 
$$|s|_p\leq\max\left\{\left|\frac{b_n}{(\lambda^{(p)}_1)^n}\right|_p:n=1,\ldots,N\right\}<1 $$ as $|b_n|_p\leq 1$ and $|(\lambda^{(p)}_1)^n|_p>1$ for every $n=1,\ldots,N$. Now $|s|_p<1$ is equivalent to say that $p$ divides $s$.
\end{proof}

\begin{lemma}\label{=p^sp}
Let $p\in\p$ be finite. Then
\begin{equation}\label{stf}
\log |1/s|_p=\sum_{|\lambda^{(p)}_i|>1}\log|\lambda^{(p)}_i|_p.
\end{equation}
\end{lemma}
\begin{proof}
%
If $|\lambda^{(p)}_1|_p\leq 1$, then $|1/s|_p= 1$ by Lemma \ref{>1}, and so both the right and the left hand side of \eqref{stf} are $0$. 
Assume that $|\lambda^{(p)}_1|_p>1$. Let $n\in\{1,\ldots,N\}$ be the largest index such that $|\lambda^{(p)}_n|_p>1$.
By \eqref{coef'} we have that $|a_j|_p\leq|a_n|_p$ for every $j\in\{1,\ldots,N\}$.
Therefore, $$|1/s|_p=\max\{1,|a_1|_p,\ldots,|a_N|_p\}=|a_n|_p=|\lambda_1^{(p)}\cdot\ldots\cdot\lambda_n^{(p)}|_p=\prod_{|\lambda_i^{(p)}|_p>1}|\lambda_i|_p .$$ Applying the logarithm we obtain the wanted equality.
\end{proof}

Since every positive integer $n$ can be written as $n=\prod_{p\in\p,p<\infty}|1/n|_p$, 
Lemma \ref{=p^sp} yields 
\begin{equation}\label{logs}
\log s=\sum_{p\in\p,p<\infty}\sum_{|\lambda^{(p)}_i|_p>1}\log|\lambda^{(p)}_i|_p.
\end{equation}

We can now verify the decomposition of $m(f(X))$ as sum of $p$-adic contributions stated in Fact C of the Introduction.

\begin{theorem}\label{m=mp}
In the above notations,
$$m(sf(X))=\sum_{p\in\p}\sum_{|\lambda^{(p)}_i|_p>1}\log|\lambda^{(p)}_i|_p.$$
\end{theorem}
\begin{proof}
By definition, 
$$m(sf(X))=\log s+\sum_{|\lambda^{(\infty)}_i|_\infty>1}\log|\lambda_i|_\infty,$$
so \eqref{logs} concludes the proof.
\end{proof}

Theorems \ref{contributi} and \ref{m=mp} give immediately the following theorem, covering the Algebraic Yuzvinski Formula. Indeed, it gives a more precise result, namely the value of the algebraic entropy of an endomorphism $\phi$ of $\Q^N$, coinciding with the Mahler measure of the characteristic polynomial of $\phi$ over $\Z$, is realized as the algebraic entropy of $\phi$ with respect to each $E_m$ with $m\in\mathcal N_1(\phi)$.

\begin{theorem}\label{H-yuz}
Let $\phi:\Q^N\to \Q^N$ be an endomorphism and $m\in\mathcal N_1(\phi)$. Then
$$h_A(\phi)=H_A(\phi,E_m)=m(p_\phi(X)),$$
where $p_\phi(X)$ is the characteristic polynomial of $\phi$ over $\Z$.
\end{theorem}

\section{Applications and open questions}\label{Appl}

In this final section of the paper, we describe the main results from \cite{DG}, \cite{DG-BT} and \cite{DG1}, which are the fundamental properties of the algebraic entropy of endomorphisms of discrete Abelian groups, underlying where the Algebraic Yuzvinski Formula is used. In general, in the proofs of these results one can separate the torsion and the torsion-free case, and reduce step by step to endomorphisms of divisible torsion-free Abelian groups of finite rank, namely, to endomorphisms of $\Q^N$ for some positive integer $N$; at this stage the Algebraic Yuzvinski Formula applies. 

Furthermore, a lot of related open problems are discussed.

\bigskip
We start introducing a subcategory of the category of morphisms, which is useful in working with entropy.
For a category $\mathfrak C$, the category $\flow(\mathfrak C)$ of \emph{flows} of $\mathfrak C$ has as objects the pairs $(G,\phi)$, where $G$ is an object of $\mathfrak C$ and $\phi\in\End_{\mathfrak C}(G)$. Moreover, a morphism in $\flow(\mathfrak C)$ between two flows $(G,\phi)$ and $(G',\phi')$ is a morphism $u: G \to G'$ in ${\mathfrak C}$ such that the diagram
\begin{equation}\label{casc-mor}
\xymatrix{
G \ar[r]^\phi\ar[d]_u & G \ar[d]^u\\
G' \ar[r]_{\phi'} & G'
}
\end{equation}
in ${\mathfrak C}$ commutes. When the domain $G$ of $\phi$ is clear and no confusion is possible, we denote a flow $(G,\phi)$ simply by $\phi$. 

\smallskip
Let now $\mathfrak{LA}$ be the category of all LCA groups, and  $\mathfrak A$ the category of all discrete Abelian groups. Usually entropy is defined as a function from the endomorphisms $\End_{\mathfrak{LA} }(G)$ of a given object $G$ of $\mathfrak{LA}$ to $\R_{\geq 0}\cup\{\infty\}$. The category of flows allows for a more precise description of entropy in categorical terms, namely we can consider $h_A(-)$ (as well as $h_T(-)$ or $h_\infty(-)$) as a function
$$h_A:\flow (\mathfrak{LA})\to \R_{\geq 0}\cup\{\infty\}\ \text{such that} \ (G,\phi)\mapsto h_A(\phi).$$
This approach makes it easier to state (and sometimes to understand) many known results.

\begin{example}
It can be deduced from classical results (see for example \cite[Chapter 12]{Ka}) that the category $\flow(\mathfrak A)$ is isomorphic to the category  $\mathrm{Mod} (\Z[X])$ of all $\Z[X]$-modules. Indeed, a $\Z[X]$-module $M_{\Z[X]}$ is exactly an Abelian group $M$, together with an endomorphism
$$\phi_X:M\to M\ \text{defined by} \ m\mapsto m\cdot X,$$ representing the action of $X$ on $M$.
\end{example}
 
A short exact sequence in $\flow(\mathfrak{LA})$ is a commutative diagram in $\mathfrak{LA}$ of the form
\begin{equation}\label{ses}
\xymatrix{0\ar[r]\ar@{=}[d]& A\ar[r]^{\alpha}\ar[d]^{\phi_1} & B\ar[r]^{\beta}\ar[d]^{\phi_2} & C \ar[r]\ar[d]^{\phi_3}& 0\ar@{=}[d]\\
0\ar[r]& A\ar[r]^{\alpha} & B\ar[r]^{\beta} & C \ar[r]& 0}
\end{equation}
where the two rows are short exact sequences in $\mathfrak{LA}$. In particular, $\alpha$ is a homomorphism which is also a topological embedding, while $\beta$ is a surjective open and continuous homomorphism.

\smallskip
Consider now a flow $(G,\phi)\in \flow(\mathfrak {LA})$. If there exists a family $\{K_i:i\in I\}$ of $\phi$-invariant closed subgroups of $G$ directed by inclusion such that $G=\bigcup_{i\in I} K_i$, then $(G,\phi)$ is the direct limit of $\{(K_i,\phi\restriction_{K_i}):i\in I\}$ in $\flow(\mathfrak{LA})$.

\begin{definition}
A function $h: \flow(\mathfrak {LA})\to \R_{\geq0}\cup\{\infty\}$ (respectively, $h: \flow(\mathfrak A)\to \R_{\geq0}\cup\{\infty\}$) is:
\begin{enumerate}[\rm (a)]
\item {\em upper continuous} if, given a flow $(G,\phi)$ in $\flow(\mathfrak {LA})$ (respectively, in $\flow(\mathfrak {A})$) that is the direct limit of a directed system of subobjects $\{(K_i,\phi\restriction_{K_i}):i\in I\}$ as above, $h_A(\phi)=\sup\{h_A(\phi\restriction_{K_i}):i\in I\}$; 
\item {\em additive} if, given a short exact sequence as in \eqref{ses} in $\flow(\mathfrak {LA})$ (respectively in $\flow(\mathfrak {A})$), $h(\phi_2)=h(\phi_1)+h(\phi_3).$
\end{enumerate}
\end{definition}

With this new terminology we write the following consequence of Lemma \ref{upp-cont}, which is a result from \cite{DG}.

\begin{corollary}\label{uc}
The function $h_A:\flow(\mathfrak{A})\to \R_{\geq 0}\cup\{\infty\}$ is upper continuous.
\end{corollary}

Note that Lemma \ref{upp-cont} is far more general than  Corollary \ref{uc}, as it holds for directed systems of invariant open subgroups of arbitrary LCA groups. Nevertheless we do not know whether it is possible to prove upper continuity of $h_A(-)$ in full generality, 
that is, the following problem remains open.

\begin{problem}
Is $h_A:\flow(\mathfrak{LA})\to \R_{\geq 0}\cup\{\infty\}$ upper continuous?
\end{problem}

The first application of the Algebraic Yuzvinski Formula is a deep result from \cite{DG}, called Addition Theorem, which shows the additivity of $h_A(-)$ when restricted to flows of $\mathfrak A$.

\begin{theorem}[Addition Theorem]\label{AT}
The function $h_A:\flow(\mathfrak{A})\to \R_{\geq 0}\cup\{\infty\}$ is additive.
\end{theorem}

In the case of endomorphisms of torsion discrete Abelian groups, Theorem \ref{AT} was proved (for $\ent(-)$) in \cite{DGSZ}. In order to extend it to the whole $\flow(\mathfrak A)$, the Algebraic Yuzvinski Formula is needed. 

The Addition Theorem for the topological entropy of endomorphisms of compact groups can be obtained as a consequence of \cite[Theorem 19]{B}; moreover, it was previously proved in the metric case in \cite{Y}.

\smallskip
Given a flow $(G,\phi)\in\flow(\mathfrak A)$, the following are particular cases of the Addition Theorem:
\begin{enumerate}[\rm (1)]
\item if $(G',\phi')\cong (G,\phi)$ then $h_A(\phi)=h_A(\phi')$;
\item if $K\leq G$ is $\phi$-invariant, then $h_A(\phi)\geq h_A(\phi\restriction_K)$;
\item if $K\leq G$ is $\phi$-invariant, then $h_A(\phi)\geq h_A(\overline\phi)$, where $\overline\phi:G/K\to G/K$ is the map induced by $\phi$.
\end{enumerate}
Note that the stability under isomorphisms in (1) is proved for endomorphisms of LCA groups in Proposition \ref{prop}(1). Moreover, the monotonicity for invariant subgroups in (2) is proved for endomorphisms of LCA groups in Lemma \ref{monotonia}(3), provided the subgroup $K$ is open. We have no counterpart in the general case of the monotonicity for quotients in (3).

\begin{problem}\label{pb-bp}
Consider $h_A:\flow(\mathfrak{LA})\to \R_{\geq0}\cup\{\infty\}$. Is it additive?

If not, is it possible at least to extend (2) or (3) above to the case when $G$ is an LCA group, $\phi\in\End(G)$ and $K$ is a $\phi$-invariant closed subgroup of $G$?
\end{problem}

We want to describe now a last instance of the Addition Theorem. Following Bowen \cite{B}, a subgroup $K$ of an LCA group $G$, is said to be {\em uniform discrete} if it is discrete and $G/K$ is compact. 

\begin{example}\label{ex-dik}
\begin{enumerate}[\rm (a)]
\item First, $\Z$ is uniform discrete in $\R$.
\item If $K_i$  is uniform discrete in $ G_i $ for  $i = 1,\ldots,N$, then $  K = K_1 \times \ldots \times K_N $ is uniform discrete in $G = G_1 \times \ldots \times G_N$  (so  $\Z^N $ is uniform discrete in  $\R^N$).
\item Moreover, $\Z$  diagonally embedded in $\R \times \prod_{p\in\p, p<\infty} \mathbb Z_p$ is uniform discrete.
\item Finally, $\Q$ diagonally embedded in the adele ring $\Q_\mathbb A$ is uniform discrete (see \cite[Lemma 4.1]{LW}).
\end{enumerate}
\end{example}

A consequence of \cite[Theorem 20]{B} is the following

\begin{fact}\label{Fact-Bowen}
If $G$ is a metrizable LCA group, $\phi$ is an endomorphism of $G$, and $K$ is a $\phi$-invariant uniform discrete subgroup of $G$, then
$$h_T(\phi)=h_T(\overline\phi),$$
where $\overline\phi:G/K\to G/K$ is the endomorphism induced by $\phi$.
\end{fact}  

Since $h_T(-)$ is trivial on discrete groups, the above fact can be viewed as a particular case of some Addition Theorem for $h_T(-)$ on LCA groups. 

\smallskip
As we have described in Section \ref{LW-sub}, Fact \ref{Fact-Bowen} and Example \ref{ex-dik}(d) were used to prove \cite[Theorem 1]{LW}, that was a fundamental step in the proof of the Yuzvinski Formula for solenoidal automorphisms given in \cite{LW}; on the other hand, we have given a direct proof of the counterpart of \cite[Theorem 1]{LW} for the algebraic entropy, that is Theorem \ref{contributi}, as it is not known whether the counterpart of Fact \ref{Fact-Bowen} holds for the algebraic entropy. In this direction, it would be interesting to answer the following question, suggested to us by Dikran Dikranjan.

\begin{problem}\label{pb-bowen}
Given $(G,\phi)\in\flow(\mathfrak{LA})$ and a uniform discrete $\phi$-invariant subgroup $K$ of $G$, is it true that $h_A(\phi)=h_A(\phi\restriction_K)$?
\end{problem}

Since $h_A(-)$ trivializes on compact Abelian groups, a positive answer to Problem \ref{pb-bp}, namely, additivity of $h_A(-)$ in the general case of flows of $\mathfrak{LA}$, would answer positively to this problem as well.

\bigskip
Our second application is to show how the Algebraic Yuzvinski Formula can be used to compute the algebraic entropy of any endomorphism of a torsion-free discrete Abelian group. The first step is to extend a given endomorphism to an endomorphism of a rational vector space with the same algebraic entropy:

\begin{lemma}\label{inviluppo}\cite[Proposition 2.12]{DG}
Let $G$ be a torsion-free discrete Abelian group and $\phi:G\to G$ an endomorphism. Denote by $D(G)=G\otimes_{\Z}\Q$ the divisible hull of $G$ and by $\phi\otimes_\Z id_\Q=\widetilde\phi:D(G)\to D(G)$ the unique extension of $\phi$ to $D(G)$. Then $h_A(\phi)=h_A(\widetilde\phi)$.
\end{lemma}

Note that the Algebraic Kolmogorov-Sinai Formula stated in the Introduction can be proved using Lemma \ref{inviluppo} and the Algebraic Yuzvinski Formula.

\medskip
We just sketch an argument for Lemma \ref{inviluppo}, in order to make this paper self-contained. Indeed, $D(G)$ is the direct limit of its $\widetilde \phi$-invariant subgroups of the form $\frac{1}{n!} G$, with $n$ ranging in $\N$. Hence, by Lemma \ref{upp-cont} we have that $h_A(\widetilde \phi)$ is the supremum of the algebraic entropies of the restrictions of $\widetilde \phi$ to each $\frac{1}{n!} G$.
Furthermore, the groups $\frac{1}{n!} G$ are all isomorphic to $G$ and the action of $\widetilde\phi$ on $\frac{1}{n!} G$ is conjugated to the action of $\phi$. Hence
$$h_A(\widetilde \phi)=\sup\left\{h_A\left(\phi\restriction_{\frac{1}{n!} G}\right):{n\in\N}\right\}=h_A(\phi),$$
by Lemma \ref{prop}(1).

\medskip
We remark that Lemma \ref{inviluppo} is the exact counterpart of \cite[Proposition 3.1]{LW} that allows the computation of the topological entropy of a solenoidal  automorphism extending it to an automorphism of a full solenoid which has the same topological entropy.

\medskip
Let $G$ be a torsion-free discrete Abelian group and let $\phi$ be an endomorphism of $G$. To compute the algebraic entropy of $\phi$, we can suppose $G$ to be a rational vector space in view of Lemma \ref{inviluppo}. 
One can realize the group $G$ as the union of a continuous chain of $\phi$-invariant subspaces
\begin{equation}\label{chain}0=K_0\subseteq K_1 \subseteq  K_2\subseteq \ldots \subseteq  K_\sigma=G\end{equation}
for some ordinal $\sigma$; denote by $\phi_{\gamma}$ the endomorphism  of $K_{\gamma+1}/K_{\gamma}$ induced by $\phi$, for all $\gamma< \sigma$. It is shown in \cite{DG} that the chain in \eqref{chain} can be constructed in such a way that either $K_{\gamma+1}/K_{\gamma}$ is finite dimensional or it is infinite dimensional and $\phi_\gamma$ is conjugated to the right Bernoulli shift $\beta_\Q$ described in Example \ref{bernoulli}, for all $\gamma<\sigma$. Additivity and upper continuity of $h_A(-)$ allow one to prove by transfinite induction that
$$h_A(\phi)=\sum_{\gamma<\sigma}h_A(\phi_\gamma),$$
where the algebraic entropy of each $\phi_\gamma$ is either infinite or it can be computed using the Algebraic Yuzvinski Formula.

\medskip
In the above discussion we could determine the algebraic entropy of an endomorphism of a torsion-free discrete Abelian group $G$ just using additivity, upper continuity and the specific values that $h_A(-)$ takes on the right Bernoulli shifts and on endomorphisms of finite dimensional rational vector spaces (given by the Algebraic Yuzvinski Formula). Using more carefully the same arguments as above, one can prove the following result from \cite{DG}.

\begin{theorem}[Uniqueness Theorem]\label{UT}
The algebraic entropy $h_A:\flow(\mathfrak A)\to\mathbb R_{\geq0}\cup\{\infty\}$ is the unique function such that:
\begin{enumerate}[\rm (1)]
\item $h_A(-)$ is additive;
\item $h_A(-)$ is upper continuous;
\item $h_{A}(\beta_K)=\log|K|$ for any finite Abelian group $K$ (where $\beta_K$ is the right Bernoulli shift defined in Example \ref{bernoulli});
\item the Algebraic Yuzvinski Formula holds for $h_A(-)$.
\end{enumerate}
\end{theorem}

The first three axioms in the Uniqueness Theorem are enough to characterize the algebraic entropy in the class of torsion discrete Abelian groups as it was proved in \cite[Theorem 6.1]{DGSZ}. More precisely, other two axioms appeared in the Uniqueness Theorem from \cite{DGSZ}, that is, invariance under conjugation and a logarithmic law; nevertheless, it can be proved that these two axioms are not necessary.

Moreover, it is worth to note that the Uniqueness Theorem for the algebraic entropy was inspired by the Uniqueness Theorem proved by Stojanov \cite{S} for the topological entropy in the class of compact (non-necessarily Abelian) groups.

\bigskip
A third application of the Algebraic Yuzvinski Formula is the following general form of the Bridge Theorem, that extends to all endomorphisms of discrete Abelian groups  both Weiss Bridge Theorem and Peters Bridge Theorem.

\begin{theorem}[Bridge Theorem]\label{BT}
 Let $(G,\phi)\in\flow(\mathfrak A)$. Then $h_A(\phi)=h_T(\widehat\phi)$.
\end{theorem}

The torsion case of this theorem is covered by Weiss Bridge Theorem. Then additivity and upper continuity of $h_A(-)$, and additivity and ``continuity on inverse limits" of $h_T(-)$, are applied to restrict to the case of automorphisms of $\Q^N$ and of its dual $\widehat{\Q}^N$. At this stage the Algebraic Yuzvinski Formula and the Yuzvinski Formula conclude the proof.

\medskip
The problem of whether it is possible to extend the Bridge Theorem to the whole $\flow(\mathfrak{LA})$ is open (see also \cite{DG-BT}):

\begin{problem}
Does the Bridge Theorem hold for every $(G,\phi)\in\flow(\mathfrak {LA})$?
Classify the flows $(G,\phi)\in\flow(\mathfrak{LA})$ such that $h_A(\phi)=h_T(\widehat\phi)$.
%
\end{problem}

\bigskip
We conclude with a last application of the Algebraic Yuzvinski Formula coming from \cite{DG1}.

\smallskip
Consider a fixed LCA group $G$, a Haar measure $\mu$ on $G$, and an endomorphism $\phi:G\to G$. For every $C\in\c(G)$ we can define a sequence 
$$\tau_C:\N_+\to \R_{\geq 0}\ \text{such that} \ n\mapsto \tau_C(n)=\mu(T_n(\phi,C)).$$
In Section \ref{Bp-ae} we said that the $\phi$-trajectory of $C$ converges exactly when the sequence $\{\frac{\log\tau_C(n)}{n}:n\in\N\}$ is convergent. Furthermore we saw in Proposition \ref{prop} that it is very useful to know whether the $\phi$-trajectory of $C$ converges for every $C\in \c(G)$. This occurs when $G$ is compact (the above sequences converge to $0$), discrete (see for example \cite[Corollary 2.2]{DG}), $G=\R^N$ or $G=\Q_p^N$ with $N$ a positive integer and $p$ a prime (use Proposition \ref{jordan-p}). We do not know if this holds in general: 

\begin{problem}
Let $(G,\phi)\in\flow(\mathfrak{LA})$ and $C\in\c(G)$. Does the $\phi$-trajectory of $C$ converge? 

If this is not true in general, classify the flows $(G,\phi)\in\flow(\mathfrak{LA})$ such that the $\phi$-trajectory of $C$ converges for every $C\in\c(G)$.
\end{problem}

Now we restrict to the context of discrete Abelian groups.
Following \cite{DG1}, fix a flow $(G,\phi)\in\flow(\mathfrak A)$; as usual we consider on $G$ the Haar measure given by the cardinality of subsets. Hence, given $F\in\c(G)$, the sequence defined above becomes
$$\tau_F:\N_+\to \R_{\geq 0}\ \text{such that}\  n\mapsto \tau_F(n)=|T_n(\phi,F)|.$$
For every $n\in\N_+$, $\tau_F(n)\leq |F|^n$, thus the sequence $\{\tau_F(n):n\in\N_+\}$ has at most exponential growth. This justifies the following definitions given in \cite{DG1}:
\begin{enumerate}[--]
\item $(G,\phi)$ has \emph{exponential growth} at $F$ if there exists $b\in\R$, $b>1$, such that $\tau_F(n)\geq b^n$ for every $n\in\N_+$;
\item $(G,\phi)$ has \emph{polynomial growth} at $F$ if there exists $P_F(X)\in\Z[X]$ such that $\tau_F(n)\leq P_F(n)$ for every $n\in\N_+$.
\end{enumerate}

If the growth of $(G,\phi)$ at $F$ is polynomial, then $H_A(\phi,F)=0$. On the other hand, if the growth of $(G,\phi)$ at $F$ is exponential, then $H_A(\phi,F)\neq 0$. Nevertheless, for an arbitrary sequence there are a lot of possible growths between polynomial and exponential. 
One of the main results from \cite{DG1} states that this is not the case for sequences of the form $\{\tau_F(n):n\in\N_+\}$:

\begin{theorem}[Dichotomy Theorem]
Let $(G,\phi)\in\flow(\mathfrak A)$ and $F\in\c(G)$. Then:
\begin{enumerate}[\rm (1)]
\item$H(\phi,F)=0$ if and only if $(G,\phi)$ has polynomial growth at $F$; 
\item$H(\phi,F)>0$ if and only if $(G,\phi)$ has exponential growth at $F$.
\end{enumerate}
In particular, $(G,\phi)$ has either exponential or polynomial growth at $F$.
\end{theorem}

 To prove this theorem the following consequence of the Algebraic Yuzvinski Formula is applied in \cite{DG1}. In particular, it is used to find non-zero periodic points of an automorphism of $\Q^N$ with zero algebraic entropy.

\begin{corollary}
Let $N$ be a positive integer and $\phi:\Q^N\to\Q^N$ an automorphism. If $h_A(\phi)=0$, then all the eigenvalues of $\phi$ are roots of unity. 
\end{corollary}

Indeed the Algebraic Yuzvinski Formula implies that the characteristic polynomial $p_\phi(X)$ of such $\phi$ over $\Z$ is monic and all the roots $\{\lambda_i:i=1,\ldots,N\}\subseteq \C$ of $p_\phi(X)$ have $|\lambda_i|\leq 1$. Now Kronecker Theorem \cite{Kr} implies that all $|\lambda_i|=1$.



\medskip
For more details about this topic see \cite{DG1}. We conclude with the following general problem.

\begin{problem}
Is it possible to extend the results from \cite{DG1} to the general case of $\flow(\mathfrak{LA})$?
\end{problem}


\begin{thebibliography}{AAAA}

\bibitem{AKM} R. L. Adler, A. G. Konheim and M. H. McAndrew, {\it Topological entropy}, Trans. Amer. Math. Soc. \textbf{114} (1965), 309--319.

\bibitem{BL} F. Blanchard and Y. Lacroix, \emph{Zero entropy factors of topological flows}, Proc. Amer. Math. Soc. \textbf{119} (1993) no. 3, 985--992.

\bibitem{B} R. Bowen, {\it Entropy for group endomorphisms and homogeneous spaces}, Trans. Amer. Math. Soc. \textbf{153} (1971), 401--414.


\bibitem{DG} D. Dikranjan and A. Giordano Bruno, {\it Entropy on abelian groups}, preprint; arXiv:1007.0533.

\bibitem{DG-BT} D. Dikranjan and A. Giordano Bruno, {\it The connection between topological and algebraic entropy}, submitted.

\bibitem{DG1} D. Dikranjan and A. Giordano Bruno, {\it The Pinsker subgroup of an algebraic flow}, Journal of Pure and Applied Algebra \textbf{216} (2012) no. 2, 364--376.

\bibitem{DGSV} D. Dikranjan, A. Giordano Bruno, L. Salce and S.Virili, {\it Intrinsic algebraic entropy}, submitted.

 
\bibitem{DGSZ} D. Dikranjan, B. Goldsmith, L. Salce and P. Zanardo, {\it Algebraic entropy for Abelian groups}, Trans. Amer. Math. Soc. \textbf{361} (2009), 3401--3434.

\bibitem{DGZ} D. Dikranjan, K. Gong and P. Zanardo, \emph{Endomorphisms of abelian groups with small algebraic entropy}, submitted.


\bibitem{DSV} D. Dikranjan, M. Sanchis and S. Virili, {\em New and old facts about entropy on uniform spaces and topological groups}, Topology and Its Applications, doi:10.1016/j.topol.2011.05.046.

\bibitem{WardLN1} M. Einsiedler and T. Ward, \emph{Ergodic Theory (with a view towards Number Theory)}, Graduate Texts in Mathematics Volume \textbf{259}, 2011.

\bibitem{Ward0} G. Everest and T. Ward, \emph{Heights of Polynomials and Entropy in Algebraic Dynamics}, Springer Verlag, 1999. 





\bibitem{HR}  E. Hewitt and K. A. Ross, {\it Abstract harmonic analysis I}, Springer-Verlag, Berlin-Heidelberg-New York, (1963).

\bibitem{Hi} E. Hironaka, \emph{What is\ldots LehmerÕs number?}, Not. Amer. Math. Soc. \textbf{56} (2009) no. 3, 374--375.

\bibitem{hood} B. M.~Hood,  {\it Topological entropy and uniform spaces}, J. London Math. Soc. (2) \textbf{(8)} (1974), 633--641. 


\bibitem{Ka} I. Kaplansky, {\it Infinite Abelian groups}, University of Michigan Publications in Mathematics, no. 2, Ann Arbor, University of Michigan Press, (1954).

\bibitem{K} N. Koblitz, {\em $p$-adic Numbers, $p$-adic Analysis, and Zeta-Functions, Second Edition}, Graduate Texts in Mathematics {\bf 58},  Springer-Verlag New York, Berlin, Heidelberg, Tokyo (1984).

\bibitem{Kr}  L. Kronecker, \emph{Zwei S\" atze \" uber Gleichungen mit ganzzahligen Coefficienten}, Jour. Reine Angew. Math. \textbf{53} (1857), 173--175.


\bibitem{Lehmer} D. H. Lehmer, \emph{Factorization of certain cyclotomic functions}, Ann. of Math. (2) \textbf{34} (1933), 461-479.

\bibitem{LW} D. A. Lind and T. Ward, \emph{Automorphisms of solenoids and $p$-adic entropy}, Ergod. Th. \& Dynam. Sys. \textbf{8} (1988), 411--419.

\bibitem{Mahler} K. Mahler, \emph{On some inequalities for polynomials in several variables}, J. London Math. Soc. \textbf{37} (1962), 341--344.

\bibitem{M} M. J. Mossinghoff, \emph{Lehmer's Problem web page}, http://www.cecm.sfu.ca/~mjm/Lehmer/lc.html.






\bibitem{Pet} J. Peters, {\it  Entropy on discrete Abelian groups}, Adv. Math. \textbf{33} (1979), 1--13.

\bibitem{Pet1}  J. Peters, \emph{Entropy of automorphisms on {L}.{C}.{A}. groups}, Pacific J. Math. \textbf{96} (1981) no. 2, 475--488.

\bibitem{Q} F. Quadros Gouva, {\it $p$-adic numbers: an introduction},  Springer-Verlag, Berlin-Heidelberg-New York, (1997).

\bibitem{RS} V. A. Rokhlin and Y. G. Sinai, \emph{Construction and properties of invariant measurable partitions}, Dokl. Akad. Nauk SSSR \textbf{141} (1961), 1038--1041. [In Russian]



\bibitem{Smyth} C. Smyth, \emph{The Mahler measure of algebraic numbers: a survey}. Number theory and polynomials, 322--349, London Math. Soc. Lecture Note Ser., 352, Cambridge Univ. Press, Cambridge, 2008.

\bibitem{S} L. N. Stojanov, {\it Uniqueness of topological entropy for endomorphisms on compact groups}, Boll. Un. Mat. Ital. B \textbf{7}  (1987) no. 3, 829--847.





\bibitem{V} S. Virili, {\it Entropy for endomorphisms of LCA groups}, to appear in Topology and Its Applications.

\bibitem{Walters} P. Walters, \emph{An Introduction to Ergodic Theory}, Springer-Verlag, New-York, 1982.

\bibitem{Ward}  T. Ward, \emph{Group automorphisms with few and with many periodic points}, Proc. Amer. Math. Soc. \textbf{133} (2004) no. 1, 91--96.

\bibitem{WardLN}  T. Ward, \emph{Entropy of Compact Group Automorphisms}, on-line Lecture Notes ({\tt http://www.mth.uea.ac.uk/$^\sim$h720/lecturenotes/}).


\bibitem{W} M. D. Weiss, {\it Algebraic and other entropies of group endomorphisms}, Math. Systems Theory \textbf{8} (1974/75) no. 3, 243--248.

\bibitem{We} A. Weil, {\it Basic Number Theory}, third edition, Springer Verlag, New York (1974).

\bibitem{Y} S. A. Yuzvinski, \emph{Metric properties of endomorphisms of compact groups}, Izv. Acad. Nauk SSSR, Ser. Mat. \textbf{29} (1965), 1295--1328 (in Russian). English Translation: Amer. Math. Soc. Transl. (2) \textbf{66} (1968), 63--98.

\bibitem{Y1} S. A. Yuzvinski, {\em Computing the entropy of a group endomorphism}, Sibirsk. Mat. Z. {\bf 8} (1967), 230--239 (in Russian). English Translation: Siberian Math. J. {\bf 8} (1968), 172--178.


\bibitem{Z-yuz} P. Zanardo, \emph{Yuzvinski's Formula}, unpublished notes.

\end{thebibliography}
\end{document}